\documentclass[11pt,reqno]{amsart}
\usepackage{indentfirst, amssymb, amsmath, amsthm, mathrsfs, setspace, indentfirst, enumerate,  mathrsfs, amsmath, amsthm}
\usepackage[bookmarksnumbered, colorlinks, plainpages]{hyperref}
\usepackage{mathrsfs}

\newtheorem*{theo2A}{Theorem 2.A}
\newtheorem*{theo2B}{Theorem 2.B}
\newtheorem*{theo2C}{Theorem 2.C}
\newtheorem*{theo2D}{Theorem 2.D}
\newtheorem*{theo2E}{Theorem 2.E}
\newtheorem*{theo2F}{Theorem 2.F}
\newtheorem*{theo2G}{Theorem 2.G}
\newtheorem*{theo2H}{Theorem 2.H}
\newtheorem*{theo2I}{Theorem 2.I}
\newtheorem*{theo2J}{Theorem 2.J}

\newtheorem*{exm2A}{Example 2.A}

\newtheorem*{conj2A}{Conjecture 2.A}

\newtheorem*{cor A}{Corollary A}
\newtheorem*{cor B}{Corollary B}
\newtheorem{ques}{Question}[section]

\newtheorem{theo}{Theorem}[section]
\newtheorem{lem}{Lemma}[section]
\newtheorem{cor}{Corollary}[section]
\newtheorem{note}{Note}[section]
\newtheorem{exm}{Example}[section]

\newcommand{\ol}{\overline}
\newcommand{\be}{\begin{equation}}
\newcommand{\ee}{\end{equation}}
\newcommand{\beas}{\begin{eqnarray*}}
\newcommand{\eeas}{\end{eqnarray*}}
\newcommand{\bea}{\begin{eqnarray}}
\newcommand{\eea}{\end{eqnarray}}

\numberwithin{equation}{section}
\begin{document}

\title[P\MakeLowercase {ower of a meromorphic function sharing value}......]{\LARGE
P\Large\MakeLowercase{ower of a meromorphic function sharing value with its
$\MakeLowercase{k}$-th order directional derivative in} $\mathbb{C}^m$}

\date{}
\author[A. B\MakeLowercase {anerjee}, S. M\MakeLowercase{ajumder} \MakeLowercase{and} D. P\MakeLowercase{ramanik}]{A\MakeLowercase {bhijit} B\MakeLowercase {anerjee}, S\MakeLowercase {ujoy} M\MakeLowercase {ajumder} \MakeLowercase {and} D\MakeLowercase {ebabrata} P\MakeLowercase {ramanik}$^*$}
\address{ Department of Mathematics, University of Kalyani, West Bengal 741235, India.}
\email{abanerjee\_kal@yahoo.co.in}
\address{Department of Mathematics, Raiganj University, Raiganj, West Bengal-733134, India.}
\email{sm05math@gmail.com, sjm@raiganjuniversity.ac.in}
\address{Department of Mathematics, Raiganj University, Raiganj, West Bengal-733134, India.}
\email{debumath07@gmail.com}

\renewcommand{\thefootnote}{}
\footnote{\emph{Corresponding Author}: Debabrata Pramanik.}

\renewcommand{\thefootnote}{\arabic{footnote}}
\setcounter{footnote}{0}

\begin{abstract} In the context of several complex variables, we investigate the uniqueness problem for a power of a meromorphic function that shares a value with its $k$-th order directional derivative in $\mathbb{C}^m$. Our results extend previous uniqueness theorems from the one-variable case to higher dimensions. Furthermore, we provide numerous examples to demonstrate that our results are, in certain senses, best possible.
\end{abstract}

\thanks{Typeset by \AmS -\LaTeX}
\maketitle

\section{{\bf Introduction}}
We define $\mathbb{Z}_+=\mathbb{Z}[0,+\infty)=\{n\in \mathbb{Z}: 0\leq n<+\infty\}$ and $\mathbb{Z}^+=\mathbb{Z}(0,+\infty)=\{n\in \mathbb{Z}: 0<n<+\infty\}$.
On $\mathbb{C}^m$, we define
\[\partial_{z_i}=\frac{\partial}{\partial z_i},\ldots, \partial_{z_i}^{l_i}=\frac{\partial^{l_i}}{\partial z_i^{l_i}}\;\;\text{and}\;\;\partial^{I}=\frac{\partial^{|I|}}{\partial z_1^{i_1}\cdots \partial z_m^{i_m}}\]
where $l_i\in \mathbb{Z}^+\;(i=1,2,\ldots,m)$ and $I=(i_1,\ldots,i_m)\in\mathbb{Z}^m_+$ is a multi-index such that $|I|=\sum_{j=1}^m i_j$.

\smallskip
Let
\[ S^{2m-1}=\{(u_1,u_2,\ldots,u_m)\in \mathbb{C}^{m}: |u_1|^2+|u_2|^2+\cdots+|u_m|^2=1\}.\]

For a differentiable function $f$ and a direction $u=(u_1,u_2,\ldots,u_m) \in S^{2m-1}$ the directional derivative of $f$ along $u$ is defined by $\partial_{u}(f)$ of $f$ along a direction  is defined by
\[\partial_u(f)=\sideset{}{_{j=1}^{m}}{\sum} u_{j}\partial_{z_j}(f).\]
\par The $k$-th order directional derivative  $\partial^k_u(f)$ of $f$ along $u$ is then defined inductively by
$$\partial^k_u(f)=\partial_u(\partial_u^{k-1}(f)),\;\;\; k\in \mathbb{Z}^+$$
with the base case  $$\partial^1_u(f)=\partial_u(f).$$

\smallskip
The classical Nevanlinna value distribution theory, developed in the 1920s by Rolf Nevanlinna, investigates the distribution of values taken by meromorphic functions of a single complex variable. With the advancement of complex analysis into higher dimensions, researchers have sought to generalize Nevanlinna's results to the setting of several complex variables. This pursuit gave rise to Nevanlinna theory in several complex variables, which examines the behavior and value distribution of holomorphic and meromorphic mappings from $\mathbb{C}^m$ into complex manifolds or projective spaces. In recent years, Nevanlinna theory in several complex variables has become a vibrant and rapidly evolving area of research. A particularly active direction involves applying value distribution theory to the study of normal families, partial differential equations (PDEs), partial difference equations and partial differential-difference equations. Researchers increasingly employ techniques from this theory to analyze the growth, value distribution and uniqueness of meromorphic solutions to such equations. These developments not only generalize classical one-variable results but also yield new insights into the qualitative behavior of complex-valued solutions in multidimensional settings (see \cite{BM}, \cite{TBC1}-\cite{PVD2}, \cite{K1}-\cite{LS1}, \cite{FL}, \cite{FL2}, \cite{Mss}, \cite{MD1}, \cite{MDP}).

\smallskip
We firstly recall some basis notions in several complex variables (see \cite{HLY,MR,WS}).
On $\mathbb{C}^m$, the exterior derivative $d$ splits $d= \partial+ \bar{\partial}$ and twists to $d^c= \frac{\iota}{4\pi}\left(\bar{\partial}- \partial\right)$. Clearly $dd^{c}= \frac{\iota}{2\pi}\partial\bar{\partial}$. A non-negative function $\tau: \mathbb{C}^m\to \mathbb{R}[0,b)\;(0<b\leq \infty)$ of class $\mathbb{C}^{\infty}$ is said to be an exhaustion of $\mathbb{C}^m$ if $\tau^{-1}(K)$ is compact whenever $K$ is. 
An exhaustion $\tau_m$ of $\mathbb{C}^m$ is defined by $\tau_m(z)=||z||^2$. The standard Kaehler metric on $\mathbb{C}^m$ is given by $\upsilon_m=dd^c\tau_m>0$. On $\mathbb{C}^m\backslash \{0\}$, we define $\omega_m=dd^c\log \tau_m\geq 0$ and $\sigma_m=d^c\log \tau_m \wedge \omega_m^{m-1}$. For any $S\subseteq \mathbb{C}^m$, let $S[r]$, $S(r)$ and $S\langle r\rangle$ be the intersection of $S$ with respectively the closed ball, the open ball, the sphere of radius $r>0$ centered at $0\in \mathbb{C}^m$.

\smallskip
Let $f$ be a holomorphic function on $G(\not=\varnothing)$, where $G$ is an open subset of $\mathbb{C}^m$. Then we can write $f(z)=\sum_{i=0}^{\infty}P_i(z-a)$, where the term $P_i(z-a)$ is either identically zero or a homogeneous polynomial of degree $i$. Certainly the zero multiplicity $\mu^0_f(a)$ of $f$ at a point $a\in G$ is defined by $\mu^0_f(a)=\min\{i:P_i(z-a)\not\equiv 0\}$.

\medskip
Let $f$ be a meromorphic function on $G$. Then there exist holomorphic functions $g$ and $h$ such that $hf=g$ on $G$ and $\dim_z h^{-1}(\{0\})\cap g^{-1}(\{0\})\leq m-2$. Therefore the $c$-multiplicity of $f$ is just $\mu^c_f=\mu^0_{g-ch}$ if $c\in\mathbb{C}$ and $\mu^c_f=\mu^0_h$ if $c=\infty$. The function $\mu^c_f: \mathbb{C}^m\to \mathbb{Z}$ is nonnegative and is called the $c$-divisor of $f$. If $f\not\equiv 0$ on each component of $G$, then $\nu=\mu_f=\mu^0_f-\mu^{\infty}_f$ is called the divisor of $f$. We define 
$\text{supp}\; \nu=\text{supp}\;\mu_f=\ol{\{z\in G: \nu(z)\neq 0\}}$.

\smallskip
Let $f$, $g$ and $a$ be meromorphic functions on $\mathbb{C}^m$. Then one can find three pairs
of entire functions $f_1$ and $f_2$, $g_1$ and $g_2$, and $a_1$ and $a_2$, in which each pair is coprime
at each point in $\mathbb{C}^m$ such that $f = f_2/f_1$, $g=g_2/g_1$ and $a = a_2/a_1$.
We say that $f$ and $g$ share $a$ CM if $\mu_{a_1f_2-a_2f_1}^0=\mu_{a_1g_2-a_2g_1}^0\;(a\not\equiv \infty)$ and $\mu_{f_1}^0=\mu_{g_1}^0\;\;(a=\infty)$. Again we say that $f$ and $g$ share $a$ IM if $\mu_{a_1f_2-a_2f_1,1}^0=\mu_{a_1g_2-a_2g_1,1}^0\;(a\not\equiv \infty)$ and $\mu_{f_1,1}^0=\mu_{g_1,1}^0\;\;(a=\infty)$.

\smallskip
A certain property is said to be true generically if it holds except at most an analytic subset of strictly smaller dimension. For example, for a meromorphic function $f$ in $\mathbb{C}^m$, we know that holomorphic functions $g$ and $h$ exist such that $\dim g^{-1}(\{0\})\cap h^{-1}(\{0\})\leq m-2$ and $f = g/h$. Consequently the common zero factor of $g$ and $h$ can be cancelled generically over its support sets.

\smallskip
For $t>0$, the counting function $n_{\nu}$ is defined by
\beas n_{\nu}(t)=t^{-2(m-1)}\int_{A[t]}\nu \upsilon_m^{m-1},\eeas
where $A=\text{supp}\;\nu$. 
The valence function of $\nu$ is defined by 
\[N_{\nu}(r)=N_{\nu}(r,r_0)=\int_{r_0}^r n_{\nu}(t)\frac{dt}{t}\;\;(r\geq r_0).\]

Also we write $N_{\mu_f^a}(r)=N(r,a;f)$ if $a\in\mathbb{C}$ and $N_{\mu_f^a}(r)=N(r,f)$ if $a=\infty$.
For $k\in\mathbb{N}$, define the truncated multiplicity functions on $\mathbb{C}^m$ by $\mu_{f,k}^a(z)=\min\{\mu_f^a(z),k\}$,
\beas \mu_{f)k}^a(z)=\begin{cases}
\mu_f^a(z), &\text{if $\mu_f^a(z)\leq  k$}\\
0, &\text{if $\mu_f^a(z)>k$}
\end{cases},\;\;
\bar{\mu}_{f)k}^a(z)=\begin{cases}
1, &\text{if $\mu_f^a(z)\leq  k$}\\
0, &\text{if $\mu_f^a(z)>k$},
\end{cases}\eeas
\beas \mu_{f(k}^a(z)=\begin{cases}
\mu_f^a(z), &\text{if $\mu_f^a(z)\geq  k$}\\
0, &\text{if $\mu_f^a(z)<k$}
\end{cases},\;\;
\bar{\mu}_{f(k}^a(z)=\begin{cases}
1, &\text{if $\mu_f^a(z)\geq  k$}\\
0, &\text{if $\mu_f^a(z)<k$}
\end{cases}\eeas
and the truncated valence functions
\beas N_{\nu}(t)=\begin{cases}
N_k(t,a;f), &\text{if $\nu=\mu_{f,k}^a$}\\
\ol{N}(t,a;f), &\text{if $\nu=\mu_{f,1}^a$}\\
N_{k)}(t,a;f), &\text{if $\nu=\mu_{f)k}^a$}\\
\ol{N}_{k)}(t,a;f), &\text{if $\nu=\bar{\mu}_{f)k}^a$}\\
N_{(k}(t,a;f), &\text{if $\nu=\mu_{f(k}^a$}\\
\ol{N}_{(k}(t,a;f), &\text{if $\nu=\bar{\mu}_{f(k}^a$}.
\end{cases}\eeas

\medskip
With the help of the positive logarithm function, we define the proximity function of $f$ by
\[m(r, f)=\mathbb{C}^m\langle r; \log^+ | f | \rangle=\int_{\mathbb{C}^m\langle r\rangle} \log^+ |f|\;\sigma_m.\]

The characteristic function of $f$ is defined by $T(r, f)=m(r,f)+N(r,f)$. We define $m(r,a;f)=m(r,f)$ if $a=\infty$ and $m(r,a;f)=m(r,1/(f-a))$ if $a$ is finite complex number. Now if $a\in\mathbb{C}$, then the first main theorem of Nevanlinna theory states that $m(r,a;f)+N(r,a;f)=T(r,f)+O(1)$, where $O(1)$ denotes a bounded function when $r$ is sufficiently large.
We define the order and the hyper-order of $f$ by
\[\rho(f):=\limsup _{r \rightarrow \infty} \frac{\log T(r, f)}{\log r}\;\text{and}\;\rho_2(f):=\limsup _{r \rightarrow \infty} \frac{\log \log T(r, f)}{\log r}.\]

Let $S(f)=\{g:\mathbb{C}^m\to\mathbb{P}^1\;\text{meromorphic}:\parallel T(r,g)=o(T(r,f))\}$, where $\parallel$ indicates that the equality holds only outside a set of finite measure on $\mathbb{R}^+$ and the element in $S(f)$ is called the small function of $f$.

\section{\bf{Uniqueness of meromorphic function concerning directional derivative}}
The question of when an entire function shares values with its derivatives-and whether this implies the function must coincide with its derivative-is a classical and extensively studied problem in value distribution theory, particularly within the framework of Nevanlinna theory. A central case of interest arises when an entire function and its first derivative share two values, counting multiplicities (CM), a scenario that has attracted considerable attention. Such uniqueness problems trace back to fundamental concerns in Nevanlinna theory about how the value-distribution properties of a function are constrained by its derivatives. A key question in this context is:
Under what conditions does an entire function coincide with its derivative if they share a certain number of values, with or without multiplicities?

Rubel and Yang \cite{RY} first considered the uniqueness of an entire function when it shares two values CM with its first derivative and proved the following

\begin{theo2A}\cite{RY} Let $f$ be a non-constant holomorphic function in $\mathbb{C}$ and let $a$ and $b$ be two distinct finite complex numbers. If $f$ and $f^{(1)}$ share $a$ and $b$ CM, then $f\equiv f^{(1)}$.
\end{theo2A}

One active line of research extends the uniqueness theory of entire and meromorphic functions to the setting of several complex variables. In particular, researchers have investigated whether analogous results hold for functions in $\mathbb{C}^m$, especially regarding their directional derivatives. In a notable contribution, Berenstein et. al. \cite{BCL} (1995) established a uniqueness theorem for non-constant meromorphic functions in $\mathbb{C}^m$ and their directional derivatives in terms of shared values. They proved the following result.

\begin{theo2B} \cite[Theorem 2.1]{BCL} Let $f$ be a non-constant meromorphic function in $\mathbb{C}^m$. If $f$ and $\partial_u(f)$ share three distinct polynomials $\alpha_j$ or $\infty$, $j=1,2,3$ CM, then $f\equiv \partial_u(f)$.
\end{theo2B}

Also in the same paper, Berenstein et al. \cite{BCL} exhibited the following example to show that the number of shared values cannot be reduced to two in Theorem 2.B.

\begin{exm2A}\cite{BCL} Let
\[f(z)=e^{e^{z_1+z_2\ldots+z_m}}\sideset{}{_{0}^{z_1}}{\int} e^{-e^{t+z_2\ldots+z_m}}\left(1-e^{z_1+z_2\ldots+z_m}\right)tdt.\]
Let $u=(1,0,\ldots,0)$, $\alpha_1(z)=z_1$ and $\alpha_2(z)=\infty$. Note that $\partial_u(f)=\partial_{z_1}(f)$ for $\dim(\mathbb{C}^m)>1$ and $\partial_u(f)=f^{(1)}$ for $\dim(\mathbb{C}^m)=1$. Since 
\[(\partial_u(f(z))-\alpha_1(z))/(f(z)-\alpha_1(z))=e^{z_1+\ldots+z_m},\]
 we see that $f$ and $\partial_u(f)$ share $\alpha_1$ and $\alpha_2$ CM.
\end{exm2A}

In 1996, Hu and Yang \cite{HY1} extended Theorem 2.B to moving targets on $\mathbb{C}^m$. Actually they proved the following result.

\begin{theo2C} \cite[Theorem 7.1]{HY1} Let $f$ be a non-constant meromorphic function in $\mathbb{C}^m$. If $f$ and $\partial_u(f)$ share $a$, $b$ and $\infty$ CM, where $a, b\in S(f)$ are distinct, then either $\partial_u (f)\equiv f$ or
\[\frac{f-a}{\partial_u(f)-a} \partial_u\left(\frac{\partial_u(f)-a}{f-a}\right)\equiv \frac{f-b}{\partial_u(f)-b} \partial_u\left(\frac{\partial_u(f)-b}{f-b}\right).\]

The latter can be ruled out by one of the following conditions:  
\[(i)\;\;\dim(\mathbb{C}^m)=1\;\;(ii)\;\;a\left(\partial_u(a)-a\right) \not \equiv 0, b\left(\partial_u(b)-b\right) \not \equiv 0, \partial_u(a) \not \equiv b\left(\text{or}\; \partial_u(b) \not \equiv a\right).\]
\end{theo2C}

Let us take $u=(u_1,\ldots,u_m)$ such that $u_j=1$ and $u_i=0$ for $i\neq j$. Clearly $\partial_u(f)=\partial_{z_j}(f)$ for $\dim(\mathbb{C}^m)>1$ and $\partial_u(f)=f^{(1)}$ for $\dim(\mathbb{C}^m)=1$. Then if $f$ and $\partial_u(f)$ share $\infty$ CM, then obviously $f$ will be a holomorphic function in $\mathbb{C}^m$. Therefore by Theorem 2.C, we get: Let $f$ be a non-constant holomorphic function in $\mathbb{C}^m$. If $f$ and $\partial_u(f)$ share $a$ and $b$ CM, where $a, b\in S(f)$ are distinct, then the conclusions of Theorem 2.C holds. Consequently Theorem 2.C is an improvement as well as an extension of Theorem 2.A from one complex variable to several complex variables.

\smallskip
In 1996, Br\"{u}ck \cite{RB} discussed the possible relation between $f$ and $f^{(1)}$ when an entire function $f$ in $\mathbb{C}$ and it's derivative $f^{(1)}$ share only one finite value CM. In this direction, Br\"{u}ck \cite{RB} proposed the following conjecture: 

\begin{conj2A}\cite{RB} Let $f$ be a non-constant entire function in $\mathbb{C}$ such that 
$\rho_{2}(f)\not\in\mathbb{N}\cup\{\infty\}$. If $f$ and $f^{(1)}$ share one finite value $a$ CM, then $f^{(1)}-a=c(f-a)$, where $c\in\mathbb{C}\backslash \{0\}$. 
\end{conj2A}

Conjecture 2.A remains a central and unresolved problem in the uniqueness theory concerning entire functions and their derivatives. Although substantial progress has been made under various additional assumptions, the conjecture in its full generality is still open. Notably, Yang and Zhang \cite{YZ} showed that the conjecture holds when, instead of the entire function itself, an appropriate power of the function is considered. They established the following result:

\begin{theo2D}\cite[Theorem 4.4]{YZ}  Let $f$ be a non-constant entire function in $\mathbb{C}$ and $n\geq 7$ be an integer. If $f^{n}$ and $(f^{n})^{(1)}$ share $1$ CM, then $f^{n}\equiv (f^{n})^{(1)}$, i.e., $f(z)=ce^{\frac{z}{n}}$, where $c\in\mathbb{C}\backslash \{0\}$. 
\end{theo2D}

In 2010, Zhang and Yang \cite{ZY} improved and generalized Theorem 2.D by considering higher order derivatives and by reducing the lower bound of the power of the entire function. In one of their results they also considered IM sharing of values. We now state two results of Zhang and Yang \cite{ZY}.

\begin{theo2E}\cite[Theorem 2.1]{ZY} Let $f$ be a non-constant entire function in $\mathbb{C}$ and $k, n$ be positive integers such that $n \geq k+1$. If $f^{n}$ and $(f^{n})^{(k)}$ share $1$ CM, then $f^{n}\equiv (f^{n})^{(k)}$, i.e., $f(z)=ce^{\frac{\lambda}{n}z}$, where $c, \lambda\in\mathbb{C}\backslash \{0\}$ such that $\lambda^{k}=1$. 
\end{theo2E} 

\begin{theo2F}\cite[Theorem 3.1]{ZY} Let $f$ be a non-constant entire function in $\mathbb{C}$ and $k, n$ be two positive integers such that $n \geq k+2$. If $f^{n}$ and $(f^{n})^{(k)}$ share $1$ IM, then conclusion of Theorem 2.E holds. 
\end{theo2F}

In connection to Theorem 2.F, Zhang and Yang \cite{ZY} posed the problem of investigating the validity of the result for $n \geq k+1$. They could prove Theorem 2.F for $n\geq k+1$ but only in the case when $k=1$. In 2023, Majumder et al. \cite{Mss} have settled the problem for $k\geq 2$. We now recall the result.

\begin{theo2G}\cite[Theorem 4.1]{Mss} Let $f$ be a transcendental entire function in $\mathbb{C}$ and let $k, n$ be two positive integers such that $n \geq k+1$. If $f^{n}$ and $(f^{n})^{(k)}$ share $1$ IM, then the conclusion of Theorem 2.F holds.
\end{theo2G} 

Let 
\[f(z)=\frac{2e^{z}+z+1}{e^{z}+1}.\]

Clearly $f$ and $f^{(1)}$ share $1$ CM, but $f^{(1)}-1\not\equiv c(f-1)$, where $c\in\mathbb{C}\backslash \{0\}$. 
Hence it is easy to conclude that Conjecture 2.A for meromorphic functions fails in general. Perhaps Yang and Zhang \cite{YZ} were the first to consider the uniqueness of a power of a meromorphic function $F=f^{n}$ and its derivative $F^{(1)}$ when they share certain value and that leads to a specific form of the function $f$. We now recall their result.
 
\begin{theo2H}\cite[Theorem 4.3]{YZ} Let $f$ be a non-constant meromorphic function in $\mathbb{C}$ and $n\geq 12$ be an integer. If $f^n$ and $(f^n)^{(1)}$ share $1$ CM, then conclusion of Theorem 2.E holds.
\end{theo2H}

In 2014, Li \cite{YL} further improved Theorem 2.H by lowering the power of the meromorphic function and proved the following result. 

\begin{theo2I}\cite[Theorem 1]{YL} Let $f$ be a non-constant meromorphic function in $\mathbb{C}$ and let $n$ and $k$ be two positive integers such that $n\geq k+2$. If $f^n$ and $(f^n)^{(k)}$ share $a(\neq 0,\infty)$ CM, then 
the conclusion of Theorem 2.E holds.
\end{theo2I} 

In the following result Zhang \cite{JLZ} generalized Theorem 2.I from sharing values CM to IM.

\begin{theo2J}\cite[Corollary 2.2]{JLZ} Let $f$ be a non-constant meromorphic function in $\mathbb{C}$ and let $n$ and $k$ be positive integers such that $n > 2k + 2$. If $f^n$ and $(f^n)^{(k)}$ share $a(\neq 0,\infty)$ IM, then conclusion of Theorem 2.E holds.
\end{theo2J}

A natural question arises as to whether analogous uniqueness results hold for powers of meromorphic functions that share a value with their $k$-th order directional derivatives in the context of several complex variables. Motivated by this inquiry, the present paper aims to generalize Theorems 2.E-2.J-originally formulated in the one-variable setting-to the framework of several complex variables. Specifically, we investigate the uniqueness properties of powers of meromorphic functions defined on $\mathbb{C}^m$ that share a finite value (either counting multiplicities (CM) or ignoring multiplicities (IM)) with their $k$-th order directional derivatives along a fixed direction. These results extend classical Nevanlinna-type uniqueness theorems to higher dimensions and contribute to the expanding body of research on value distribution and uniqueness theory in several complex variables. 

Now we state our results.

\begin{theo}\label{t2.1} Let $f$ be a non-constant entire function in $\mathbb{C}^m$ and $k, n$ be positive integers such that $n\geq k+1$ and $\partial_u^k(f^n)\not\equiv 0$. If $f^n$ and $\partial_u^k(f^n)$ share $1$ CM, then $f\equiv c\partial_u(f)$, where $c\in\mathbb{C}\backslash \{0\}$ such that $(c/n)^k=1$. In particular if $u_j=1$ and $u_i=0$ for $i\neq j$, then 
\[f(z)=e^{cz_j+\alpha(z)},\]
where $\alpha(z)=\alpha(z_1,\ldots,z_{j-1}, z_{j+1},\ldots,z_m)$ is a non-constant entire function in $\mathbb{C}^{m-1}$ and $(c/n)^k=1$.
\end{theo}

\begin{theo}\label{t2.3} Let $f$ be a non-constant entire function in $\mathbb{C}^m$ and $k, n$ be positive integers such that $n\geq k+2$ and $\partial_u^k(f^n)\not\equiv 0$. If $f^n$ and $\partial_u^k(f^n)$ share $1$ IM, then conclusion of Theorem \ref{t2.1} holds.
\end{theo}

Following corollary can be derived from Theorem \ref{t2.3}.

\begin{cor}\label{c2.1} Let $f$ be a non-constant entire function in $\mathbb{C}^m$ such that $\ol N_{2)}(r,0;f)=o(T(r,f))$ and let $k, n$ be two positive integers such that $n\geq k$ and $\partial_u^k(f^n)\not\equiv 0$. If $f^n$ and $\partial_u^k(f^n)$ share $1$ IM, then the conclusion of Theorem \ref{t2.1} holds.
\end{cor}

\begin{note} If $k \geq 2$, then in Corollary \ref{c2.1} instead of $\ol N_{2)}(r,0;f)=o(T(r,f))$ we can assume $N_{1)}(r,0;f)=o(T(r,f))$.
\end{note}

Now, in view of Theorems Theorems 2.G and  \ref{t2.3}, it is natural to pose the following question:
\begin{ques} If $m\geq 2$ and $n\geq k+1$, is it possible to establish an analogue of Theorem \ref{t2.3} in the setting of several complex variables?

\end{ques}

In this paper, we provide a partial answer to this question in the case $k=1$, by proving the following result.

\begin{theo}\label{t2.2} Let $f$ be a non-constant entire function in $\mathbb{C}^m$ and $n\geq 2$ be an integer such that $\partial_u(f^n)\not\equiv 0$. If $f^n$ and $\partial_u(f^n)$ share $1$ IM, then $f\equiv n\partial_u(f)$. In particular if $u_j=1$ and $u_i=0$ for $i\neq j$, then 
\[f(z)=e^{cz_j+\alpha(z)},\]
where $\alpha(z)=\alpha(z_1,\ldots,z_{j-1}, z_{j+1},\ldots,z_m)$ is a non-constant entire function in $\mathbb{C}^{m-1}$.
\end{theo}

The following theorems are natural extensions of Theorems 2.I and 2.J, respectively, generalizing the classical results from the setting of one complex variable to the richer framework of several complex variables.

\begin{theo}\label{t2.4} Let $f$ be a non-constant meromorphic function in $\mathbb{C}^m$ and let $k$ and $n$ be two positive integers such that $n\geq k+2$ and $\partial_u^k(f^n)\not\equiv 0$. If $f^n$ and $\partial_u^k(f^n)$ share $1$ CM, then $f^n\equiv \partial_u(f^n)$. In particular if $f$ and $\partial_u(f)$ share $\infty$ CM, then $f\equiv c\partial_u(f)$, where $(c/n)^k=1$. Furthermore if $u_j=1$ and $u_i=0$ for $i\neq j$, then 
\[f(z)=e^{cz_j+\alpha(z)},\]
where $\alpha(z)=\alpha(z_1,\ldots,z_{j-1}, z_{j+1},\ldots,z_m)$ is a non-constant entire function in $\mathbb{C}^{m-1}$.
\end{theo}
\begin{theo}\label{t2.5} Let $f$ be a non-constant meromorphic function in $\mathbb{C}^m$ and let $k, n$ be two positive integers such that $n\geq 2k+3$ and $\partial_u^k(f^n)\not\equiv 0$. If $f^n$ and $\partial_u^k(f^n)$ share $1$ IM, then the conclusion of Theorem \ref{t2.4} holds.
\end{theo}
The table below presents all Theorems and one Corollary in a concise format, summarizing their statements and conclusions. \newpage 
\begin{table}[h!]
	\centering
	\setlength{\tabcolsep}{5pt} 
	\renewcommand{\arraystretch}{1.5} 
	\begin{tabular}{|c|p{7.0cm}|p{5cm}|}
		\hline
		Theorem No. & \hspace{5cc}Statement & \hspace{1cc}Conclusion / Result \\
		\hline
		Theorem 2.1 & $f$ non-constant entire in $\mathbb{C}^m$; $k,n$ positive integers; $n \ge k+1$; $\partial_u^k(f^n) \not\equiv 0$; $f^n$ and $\partial_u^k(f^n)$ share 1 CM. & $f \equiv c \, \partial_u(f)$, $c \in \mathbb{C}\setminus\{0\}$, $(c/n)^k = 1$. If $u_j=1$ and $u_i=0$ for $i\neq j$: $f(z)=e^{c z_j + \alpha(z)}$, $\alpha(z)$ non-constant entire in $\mathbb{C}^{m-1}$. \\
		\hline
		Theorem 2.2 & $f$ non-constant entire in $\mathbb{C}^m$; $k,n$ positive integers; $n \ge k+2$; $\partial_u^k(f^n)\not\equiv 0$; $f^n$ and $\partial_u^k(f^n)$ share 1 IM. & Conclusion of Theorem 2.1 holds. \\
		\hline
		Corollary 2.1 & $f$ non-constant entire in $\mathbb{C}^m$; ${\overline N_{2)}(r,0;f)} = o(T(r,f))$; $k,n$ positive integers; $n \ge k$; $\partial_u^k(f^n)\not\equiv 0$; $f^n$ and $\partial_u^k(f^n)$ share 1 IM. & Conclusion of Theorem 2.1 holds. \\
		\hline
		Theorem 2.3 & $f$ non-constant entire in $\mathbb{C}^m$; $n \ge 2$; $\partial_u(f^n)\not\equiv 0$; $f^n$ and $\partial_u(f^n)$ share 1 IM. & $f \equiv n \, \partial_u(f)$. If $u_j=1$ and $u_i=0$ for $i\neq j$: $f(z)=e^{c z_j + \alpha(z)}$, $\alpha(z)$ non-constant entire in $\mathbb{C}^{m-1}$. \\
		\hline
		Theorem 2.4 & $f$ non-constant meromorphic in $\mathbb{C}^m$; $k,n$ positive integers; $n \ge k+2$; $\partial_u^k(f^n)\not\equiv 0$; $f^n$ and $\partial_u^k(f^n)$ share 1 CM. & $f^n \equiv \partial_u(f^n)$; if $f$ and $\partial_u(f)$ share $\infty$ CM, then $f \equiv c\,\partial_u(f)$, $(c/n)^k = 1$. If $u_j=1$ and $u_i=0$ for $i\neq j$: $f(z)=e^{c z_j + \alpha(z)}$, $\alpha(z)$ non-constant entire in $\mathbb{C}^{m-1}$. \\
		\hline
		Theorem 2.5 & $f$ non-constant meromorphic in $\mathbb{C}^m$; $k,n$ positive integers; $n \ge 2k+3$; $\partial_u^k(f^n)\not\equiv 0$; $f^n$ and $\partial_u^k(f^n)$ share 1 IM. & Conclusion of Theorem 2.4 holds. \\
		\hline
	\end{tabular}\vspace{1.5cc}
	\caption{Theorems and Corollary.}
\end{table}
The following ensures the necessity of the condition ``$n\geq k+1$'' in Theorem \ref{t2.1} and Corollary \ref{c2.1}.
\begin{exm}\label{exm2.1} Let 
\[f(z)=2e^{\frac{z_1+\ldots +z_m}{2}}-1\]
and $k=n=1$. Let us choose $u=(u_1,\ldots,u_m)$ in such a way that $u_1+\ldots+u_m=1$. Therefore $\partial_u(f)=e^{\frac{z_1+\ldots +z_m}{2}}$. It is easy to verify that $f$ and $\partial_u(f)$ share $1$ CM, but conclusion of Theorem \ref{t2.1} cease to hold.
\end{exm}

\begin{exm}
	Let $f(z) = e^{2(z_1 + \dots + z_m)} +\frac{1}{2}, \quad n = k = 1$, $u=(u_1,\ldots,u_m)$, such that $u_1+\ldots+u_m=1$.
	Then the directional derivative is
	$\partial_u f = \frac{\partial f}{\partial z_1} = 2e^{z_1 + \dots + z_m}.$
	Observe that $f$ and $\partial_u f$ share the value $1$ CM, but conclusion of  Theorem \ref{t2.1} dose not hold.
		\end{exm}
		The next examples verify the necessity for $k=2$.
		\begin{exm} Take $n=1$. Let $f(z) = e^{\sqrt{2}\,(z_1 + z_2 + \cdots + z_m)} + \frac{1}{2}$, $\qquad z = (z_1,\ldots,z_m) \in \mathbb{C}^m$,  $m \ge 1$. Choose $u = (1,0,0,\ldots,0)$. Then $	\partial_u f = \sqrt{2}\, e^{\sqrt{2}(z_1+\cdots+z_m)}, \qquad
			\partial_u^2 f = 2\, e^{\sqrt{2}(z_1+\cdots+z_m)} =  2f - 1.	$
			Hence $\partial_u^2 f - 1 = 2(f - 1)$ and so it follows that $f$ and $\partial_u^2 f$ share the value $1$ CM.
			However, Theorem \ref{t2.1} fails.
		\end{exm} 

		\begin{exm}	Take $n = 2$. Let	$f(z) = \sqrt{\frac{5}{2}}\;\sin z_1$, $u = (1,0,\dots,0)$, $m\ge 2$. 	
			Then $f^{2} =\frac{5}{2} \sin^2 z_1$, so that  $\partial_u^2(f^2) = 5\cos 2z_1 = 5-10\sin^2 z_1.$ Thus $f^2$ and $\partial_u^2(f^2)$ share $1$ CM.
			However, the conclusion of Theorem \ref{t2.1} cease to hold.			
		\end{exm}	
		\begin{exm}	Take $n = 2$. Let $f(z) = e^{z_1+z_2+\ldots+z_m}+\frac{3}{8}e^{-(z_1+z_2+\ldots+z_m)}$, $u = (1,0,\dots,0)$, $m\ge 2$. 	
		Then $f^{2} =e^{2(z_1+z_2+\ldots+z_m)}+\frac{9}{64}\;e^{-2(z_1+z_2+\ldots+z_m)}+\frac{3}{4}$, so that  $\partial_u^2(f^2) = 4e^{2(z_1+z_2+\ldots+z_m)}+\frac{9}{16}\;e^{-2(z_1+z_2+\ldots+z_m)}$. Thus $f^2$ and $\partial_u^2(f^2)$ share $1$ CM. 
		However, the conclusion of Theorem \ref{t2.1} cease to hold.		
		\end{exm}			
The following examples show that Theorem \ref{t2.1} does not hold for a non-constant meromorphic function in $\mathbb{C}^m$.
\begin{exm} Let 
\[f(z)=\frac{z_1+z_2+\ldots+z_m}{1-e^{-(z_1+z_2+\ldots+z_m)}},\]
$n=k=1$ and $u=(1,0,\ldots,0)$.
Note that
\[\partial_u(f(z))=\frac{\partial f(z)}{\partial z_1}=\frac{1-(z_1+z_2+\ldots+z_m+1)e^{-(z_1+z_2+\ldots+z_m)}}{\left(1-e^{-(z_1+z_2+\ldots+z_m)}\right)^2}\]
and so
\beas f(z)-1=\frac{z_1+z_2+\ldots+z_m-1+e^{-(z_1+z_2+\ldots+z_m)}}{1-e^{-(z_1+z_2+\ldots+z_m)}},\eeas
\beas \partial_u(f(z))-1=-e^{-(z_1+z_2+\ldots+z_m)}\;\;\frac{z_1+z_2+\ldots+z_m-1+e^{-(z_1+z_2+\ldots+z_m)}}{\left(1-e^{-(z_1+z_2+\ldots+z_m)}\right)^{2}}.\eeas

Clearly $f$ and $\partial_u(f)$ share $1$ CM, but $f\not\equiv \partial_u(f)$.
\end{exm}

\begin{exm}
	Let 
	$
	f(z) = \frac{z_1 + z_2 + \dots + z_m}{1 + z_1 + \dots + z_m}$, $ n = k = 1, u = (1,0,\dots,0).$
	Then the directional derivative is
	$
	\partial_u f = \frac{\partial f}{\partial z_1} = \frac{1}{(1 + z_1 + \dots + z_m)^2}.$
	Clearly, \(f\) and \(\partial_u f\) share the value \(1\) CM, but $
	f \not\equiv \partial_u f.$
	\end{exm}
\begin{exm}
	Let 
$$	f(z) = \frac{e^{z_1 + \dots + z_m} - 1}{z_1 + \dots + z_m}, \quad n = k = 1, \quad u = (1,0,\dots,0).
	$$
	Then the directional derivative is
	$
	\partial_u f(z) = \frac{\partial f}{\partial z_1} = \frac{(z_1 + \dots + z_m - 1)e^{z_1 + \dots + z_m} + 1}{(z_1 + \dots + z_m)^2}.
	$
	
	We note that
	\[
	f(z) - 1 = \frac{e^{z_1 + \dots + z_m} - 1 - (z_1 + \dots + z_m)}{z_1 + \dots + z_m},
	\]
	\[
	\partial_u f(z) - 1 = \frac{(z_1 + \dots + z_m - 1)e^{z_1 + \dots + z_m} + 1 - (z_1 + \dots + z_m)^2}{(z_1 + \dots + z_m)^2}.
	\]
	
	Clearly, $f$ is non-constant and meromorphic in $\mathbb{C}^m$ such that $f$ and $\partial_u f(z) - 1 $ share CM but $f \not\equiv \partial_u f$.
\end{exm}

The condition ``$\ol N_{2)}(r,0;f)=o(T(r,f))$'' in Corollary \ref{c2.1} is 
necessary as shown below.

\begin{exm} Let 
\[f(z)=1+\tan (z_1+\ldots+z_m)\]
and $u=(u_1,\ldots,u_m)$ such that $u_1+\ldots+u_m=1$. Note that $\partial_u(f)-1=(f(z)-1)^2$
and so $f$ and $\partial_u(f)$ share $1$ IM. Since $1 \pm i$ are the Picard exceptional values of $f$, we have 
$\ol N_{2)}(r,0;f)\not=o(T(r, f))$. Clearly $f \not\equiv \partial_u(f)$. 
\end{exm}

\subsection {{\bf Auxiliary lemmas}}
In the proof of Theorems \ref{t2.1}-\ref{t2.5}, we use the following key lemmas.
First we recall the lemma of logarithmic derivative:
\begin{lem}\cite[Lemma 1.37]{HLY}\label{l2} Let $f$ be a non-constant meromorphic function in $\mathbb{C}^m$ and   $I=(\alpha_1,\ldots,\alpha_m)\in \mathbb{Z}^m_+$ be a multi-index. Then for any $\varepsilon>0$, we have
\[m\Big(r,\frac{\partial^I(f)}{f}\Big)\leq |I|\log^+T(r,f)+|I|(1+\varepsilon)\log^+\log T(r,f)+O(1)\]
for all large $r$ outside a set $E$ with $\int_E d\log r<\infty$.
\end{lem}

The following result is known as second main theorem:
\begin{lem}\cite[Lemma 1.2]{HY1}\label{l1} Let $f$ be a non-constant meromorphic function in $\mathbb{C}^m$ and let $a_1,a_2,\ldots,a_q$ be different points in $\mathbb{P}^1$. Then
\beas \parallel (q-2)T(r,f)\leq \sideset{}{_{j=1}^{q}}{\sum} \ol N(r,a_j;f)+O(\log (rT(r,f))).\eeas
\end{lem}

\begin{lem}\label{l2a} \cite[Theorem 1.26]{HLY} Let $f$ be a non-constant meromorphic function in $\mathbb{C}^m$. Assume that 
$R(z, w)=\frac{A(z, w)}{B(z, w)}$. Then
\beas T\left(r, R_f\right)=\max \{p, q\} T(r, f)+O\Big(\sideset{}{_{j=0}^{p}}{\sum}T(r, a_j)+\sideset{}{_{j=0}^{q}}{\sum}T(r, b_j)\Big),\eeas
where $R_f(z)=R(z, f(z))$ and two coprime polynomials $A(z, w)$ and $B(z,w)$ are given
respectively $A(z,w)=\sum_{j=0}^p a_j(z)w^j$ and $B(z,w)=\sum_{j=0}^q b_j(z)w^j$.
\end{lem}

\begin{lem}\label{l3} \cite[Lemma 2.1]{HY2} Let $f$ be a non-constant meromorphic function in $\mathbb{C}^m$. Take a positive integer $n$ and take polynomials of $f$ and its partial derivatives:
\beas\label{cl1} P(f)=\sideset{}{_{\mathbf{p} \in I}}{\sum} a_{\mathbf{p}} f^{p_0}\big(\partial^{\mathbf{i}_1} f\big)^{p_1} \cdots\big(\partial^{\mathbf{i}_l} f\big)^{p_l}, \quad \mathbf{p}=\left(p_0, \ldots, p_l\right) \in \mathbb{Z}_{+}^{l+1},\eeas
\beas\label{cl2} Q(f)=\sideset{}{_{\mathbf{q} \in J}}{\sum} c_{\mathbf{q}} f^{q_0}\big(\partial^{\mathbf{j}_1} f\big)^{q_1} \cdots\big(\partial^{\mathbf{j}_s} f\big)^{q_s}, \quad \mathbf{q}=\left(q_0, \ldots, q_s\right) \in \mathbb{Z}_{+}^{s+1}\eeas
and $B(f)=\sideset{}{_{k=0}^{n}}{\sum} b_k f^k,$
where $I, J$ are finite sets of distinct elements and $a_{\mathbf{p}}, c_{\mathbf{q}}, b_k$ are meromorphic functions on $\mathbb{C}^m$ such that $\parallel T(r,a_{\mathbf{p}})=o(T(r,f))$, $\parallel T(r,c_{\mathbf{q}})=o(T(r,f))$, $\parallel T(r,b_k)=o(T(r,f))$ and $b_n \not \equiv 0$. Assume that $f$ satisfies the equation $B(f) P(f)=Q(f)$.
If $\deg(Q(f)) \leq n=\deg(B(f))$, then $\parallel\; m(r, P(f))=o(T(r, f))$.
\end{lem}

\subsection {{\bf Proofs of the main results}} 
\begin{proof}[Proof of Theorem \ref{t2.1}] 
Let $F = f^{n}$. By the given conditions, $F$ and $\partial_u^k(F)$ share $1$ CM and $n\geq k+1$. Using Lemma \ref{l2}, we have  
$\parallel T(r,\partial_u^k(F))=O(T(r,f))$. Note that
\begin{align}\label{bm.1a} \partial_u^k(F)&=\partial_u^k(f^n)\\&=
\partial_u^{k-1}\big(u_1\partial_{z_1}(f^n)+u_2\partial_{z_2}(f^n)+\ldots +u_m\partial_{z_m}(f^n)\big)\nonumber\\
&= \partial_u^{k-1}\big(nf^{n-1}u_1\partial_{z_1}(f)+nf^{n-1}u_2\partial_{z_2}(f)+\ldots +nf^{n-1}u_m\partial_{z_m}(f)\big)\nonumber
\\& = \partial_u^{k-1}\big( nf^{n-1}\partial_u(f)\big) \nonumber
\\ &=  n\partial_u^{k-2}\big( (n-1)f^{n-2}(\partial_u(f))^{2} + f^{n-1}\partial_u^2(f)\big) \nonumber 
\\ &= n(n-1)\partial_u^{k-3}\big( (n-2)f^{n-3}(\partial_u(f))^{3}\big)+ n(n-1)\partial_u^{k-3}\big( 2f^{n-2}\partial_u(f)\partial_u^2(f)\big) 
\nonumber \\ & 
+n\partial_u^{k-3}\big( (n-1)f^{n-2}\partial_u(f)\partial_u^2(f)\big) + n\partial_u^{k-3}\big(f^{n-1}\partial_u^3(f)\big) 
\nonumber \\ & = \cdots \cdots \cdots \quad\cdots\cdots \quad\cdots\cdots\nonumber \\ & =
\frac{n!}{(n-k)!}f^{n-k}(\partial_u(f))^{k}+\frac{k(k-1)}{2}\frac{n!}{(n-k+1)!}f^{n-k+1}(\partial_u(f))^{k-2}\partial_u^2(f)\nonumber \\ &+\cdots+nf^{n-1}\partial_u^k(f)\nonumber\\&=
\frac{n!}{(n-k)!}f^{n-k}(\partial_u(f))^{k}+\frac{k(k-1)}{2}\frac{n!}{(n-k+1)!}f^{n-k+1}(\partial_u(f))^{k-2}\partial_u^2(f)+R_1(f),\nonumber
\end{align}
where $R_{1}(f)$ is a differential polynomial in $f$ and each term of $R_{1}(f)$ contains $f^{l} (3 \leq l \leq n-1$) as a factor.
Let
\begin{align}\label{bm.1} \varphi = \frac{\partial_u(F)(F - \partial_u^k(F))}{F(F - 1)}.\end{align}

Now we consider the following two cases.\par

\medskip
{\bf Case 1.} Let $\varphi \not\equiv 0$. Then $F\not\equiv \partial_u^k(F)$. Also from (\ref{bm.1}), we get 
\begin{align*}\label{bm.2} \varphi = \frac{\partial_u(F)}{F - 1}\Big(1 - \frac{\partial_u^k(F)}{F}\Big)
\end{align*}
and so applying Lemma \ref{l2} get $\parallel m(r,\varphi)=o(T(r,F))$. Since $\parallel T(r,F)=nT(r,f)+O(1)$, it follows that $\parallel m(r,\varphi)=o(T(r,f))$. 

Take $z_0\in \mathbb{C}^m$. Suppose $\mu_F^1(z_0)>0$. Since $\operatorname{supp} \mu_F^1=\operatorname{supp}\mu_{\partial_u^k(F)}^1$, we have $\mu_{\partial_u^k(F)}^1(z_0)>0$. By the given condition, we get $\mu^1_F(z_0)=\mu^1_{\partial_u^k(F)}(z_0)$. Then from (\ref{bm.1}), we have generically $\mu_{\varphi}^0(z_0)\geq 0$
over $\text{supp}\;\mu_F^1$. Hence $\mu_{\varphi}^{\infty}(z)=0$ generically for all $z\in\mathbb{C}^m$ and so $N(r,\varphi)=0$. Since $\parallel m(r,\varphi)=o(T(r,f))$, we have $\parallel T(r,\varphi)=o(T(r,f))$.
Using the first main theorem, we get $\parallel T\big(r,\frac{1}{\varphi}\big)=o(T(r,f))$.
Now from (\ref{bm.1}), we get 
\bea\label{bm.4}\frac{1}{F} = \frac{1}{\varphi}\Big(\frac{\partial_u(F)}{F-1} - \frac{\partial_u(F)}{F}\Big)\Big(1-\frac{\partial_u^k(F)}{F}\Big).
\eea

Applying Lemma \ref{l2} to (\ref{bm.4}) we get $\parallel m(r,0;F)=o(T(r,f))$ and so $\parallel m(r,0;f)=o(T(r,f))$.

Now we divide following two sub-cases.\par

\medskip
{\bf Sub-case 1.1} Let $n\geq k+2$. Take $z_0\in \mathbb{C}^m$. Suppose $\mu_f^0(z_0)>0$. Obviously $\mu_F^0(z_0)=n\mu^0_f(z_0)$ over $\text{supp}\;\mu_f^0$. Also from (\ref{bm.1a}), we get 
$\mu_{\partial_u^k(F)}^0(z_0)\geq (n-k)\mu_f^0(z_0)$ over $\text{supp}\;\mu_f^0$. Then from (\ref{bm.1}), we have generically
\bea\label{bm.3b} \mu_{\varphi}^0(z_0)\geq (n-k-1)\mu_f^0(z_0)\eea
over $\text{supp}\;\mu_f^0$. 
Then (\ref{bm.3b}) gives $\mu_{\varphi}^0(z_0)\geq \mu_f^0(z_0)$ over $\text{supp}\;\mu_f^0$. Consequently $N(r,0;f)\leq N(r,0;\varphi)$ and so first main theorem yields 
\begin{align*} \parallel N(r,0;f)\leq T(r,\varphi)+O(1)=o(T(r,f)),
\end{align*}
i.e., $\parallel N(r,0;f)=o(T(r,f))$. Since $\parallel m(r,0;f)=o(T(r,f))$ we get $\parallel T\big(r,\frac{1}{f}\big)=o(T(r,f))$, i.e., $\parallel T(r,f)=o(T(r,f))$, which is impossible.\par

\medskip
{\bf Sub-case 1.2.} Let $n=k+1$. Now from (\ref{bm.1a}), we have 
\bea\label{bm.5} \partial_u^k(F)=\partial_u^{k}(f^{k+1})&=& 
(k+1)!f(\partial_u(f))^{k} + \frac{k(k-1)}{4}(k+1)!f^{2}(\partial_u(f))^{k-2}\partial_u^2(f)\\ &&
+\ldots+(k+1)f^k\partial_u^k(f)\nonumber\\ &=&
(k+1)!f(\partial_u(f))^{k} + \frac{k(k-1)}{4}(k+1)!f^{2}(\partial_u(f))^{k-2}\partial_u^2(f)+R_1(f),\nonumber
\eea
where $R_{1}(f)$ is a differential polynomial in $f$ and each term of $R_{1}(f)$ contains $f^{l} (3 \leq l \leq k$) as a factor.

We assume that $\mu^0_{f(2}(z_0)> 0$ for $z_0\in \mathbb{C}^m$. Obviously $\mu_F^0(z_0)=(k+1)\mu^0_{f(2}$ over $\text{supp}\;\mu_f^0$. Also from (\ref{bm.5}), we get 
\[\mu_{\partial_u^k(F)}^0(z_0)\geq (k+1)\mu^0_{f(2}-k\]
over $\text{supp}\;\mu_f^0$. Similarly
\[\mu_{\partial_u(F)}^0(z_0)\geq (k+1)\mu^0_{f(2}-1\]
over $\text{supp}\;\mu_f^0$. Then from (\ref{bm.1}), we have generically
\bea\label{bm.3} \mu_{\varphi}^0(z_0)\geq (k+1)\big(\mu_{f(2}^0(z_0)-1\big)\eea
over $\text{supp}\;\mu_f^0$. Consequently $N_{(2}(r,0;f)\leq N(r,0;\varphi)$ and so by the first main theorem, we get 
\[\parallel N_{(2}(r,0;f)\leq T(r,\varphi)+O(1)=o(T(r,f)),\]
i.e., 
\bea\label{bm.5a}\parallel\;N_{(2}(r,0;f)=o(T(r,f)).\eea

Since $\parallel m(r,0;f)=o(T(r,f))$, we get 
\bea\label{bm.6}\parallel\;T(r,f)=N_{1)}(r,0;f)+o(T(r,f)).\eea

From (\ref{bm.6}) it is clear that $\parallel N_{1)}(r,0;f)\neq o(T(r,f))$. Again from (\ref{bm.5}), we get
\bea\label{bm.7} \partial_u^{k+1}(F)=
(k + 1)!(\partial_u(f))^{k + 1} + \frac{k(k + 1)}{2}(k + 1)!f(\partial_u(f))^{k - 1}\partial_u^2(f) +R_2(f),\eea  
where $R_{2}(f)$ is a differential polynomial in $f$.
Since $F$ and $\partial_u^k(F)$ share $1$ CM, then there exists an entire function $\alpha$ in $\mathbb{C}^m$ such that
\bea\label{bm.8} \partial_u^k(F)-1= {e}^{\alpha} (F-1).\eea

\smallskip
Let $\alpha$ be a constant and $e^{\alpha}=d$. Then from (\ref{bm.8}), we get 
\[\partial_u^k(F)-1=d(F-1).\]

Let us take
$\mu^0_{f)1}(z_0)> 0$ for $z_0\in \mathbb{C}^m$. Obviously $\mu^0_F(z_0)>0$ and $\mu^0_{\partial_u^k(F)}(z_0)>0$ over $\text{supp}\;\mu_f^0$. Then it is easy to deduce that $d=1$ and so $F\equiv \partial_u^k(F)$, which is impossible.

\smallskip
Let $e^{\alpha}$ be non-constant. Clearly (\ref{bm.8}) gives
\[\parallel T(r,e^{\alpha})\leq T(r,\partial_u^k(F))+T(r,F)+O(1)\leq 2nT(r,f)+o(T(r,f))\]
and so $\parallel T(r,e^{\alpha})=O(T(r,f))$. Now using Lemma \ref{l2}, we get
\[\parallel\;T(r,\partial_u(\alpha))=m(r,\partial_u(\alpha))=m\Big(r,\frac{\partial_u(e^{\alpha})}{e^{\alpha}}\Big)=o(T(r,e^{\alpha}))\]
and so $\parallel T(r,\partial_u(\alpha))=o(T(r,f))$. 
Again from (\ref{bm.8}), we have
\bea\label{bm.9}\partial_u^{k+1}(F)=e^\alpha (F-1)\partial_u(\alpha)+e^{\alpha} \partial_u(F).\eea

Combining (\ref{bm.8}) with (\ref{bm.9}), we get
\[\frac{\partial_u^{k+1}(F)}{\partial_u^{k}(F)-1}=\partial_u(\alpha)+\frac{\partial_u(F)}{F-1},\]
i.e., 
\[\partial_u^{k+1}(F) F-\partial_u(\alpha) \partial_u^k(F) F-\partial_u^k(F) \partial_u(F)=\partial_u^{k+1}(F)-\partial_u(\alpha)\big(\partial_u^k(F)+F\big)-\partial_u(F)+\partial_u(\alpha)\]
and so
\bea\label{bm.10} f^{k+1}P(f)=Q(f),\eea
where
\bea\label{bm.11} P(f)=\partial_u^{k+1}(F)-\partial_u(\alpha) \partial_u^k(F)-\partial_u^k(F)\frac{\partial_u(F)}{F}\eea
and 
\[Q(f)=\partial_u^{k+1}(F)-\partial_u(\alpha)\big(\partial_u^k(F)+F\big)-\partial_u(F)+\partial_u(\alpha).\]

Now using (\ref{bm.5}) and (\ref{bm.7}) to (\ref{bm.11}), we get 
\bea\label{bm.13} &&P(f)\\&=&\partial_u^{k+1}(F)-\partial_u(\alpha) \partial_u^k(F)-\partial_u^k(F)\frac{\partial_u(F)}{F}\nonumber\\&=&
(k + 1)!(\partial_u(f))^{k + 1} + \frac{k(k + 1)}{2}(k + 1)!f(\partial_u(f))^{k - 1}\partial_u^2(f) +R_2(f)\nonumber\\&&
-\partial_u(\alpha)\left\lbrace (k+1)!f(\partial_u(f))^{k} + \frac{k(k-1)}{4}(k+1)!f^{2}(\partial_u(f))^{k-2}\partial_u^2(f)+R_1(f) \right\rbrace\nonumber\\&&-(k+1)\frac{\partial_u(f)}{f}\left\lbrace (k+1)!f(\partial_u(f))^{k} + \frac{k(k-1)}{4}(k+1)!f^{2}(\partial_u(f))^{k-2}\partial_u^2(f)+R_1(f) \right\rbrace\nonumber\\&=&
-k(k + 1)!(\partial_u(f))^{k + 1}-\frac{k(k+1)(k-3)}{4}(k + 1)!f(\partial_u(f))^{k - 1}\partial_u^2(f)\nonumber\\&&-(k+1)!f(\partial_u(f))^{k}\partial_u(\alpha)+R_3(f),\nonumber
\eea
where $R_{3}(f)$ is a differential polynomial in $f$ such that each term of $R_3(f)$ contains $f^l$
for some $l\;(1\leq l\leq k)$ as a factor. Again using (\ref{bm.5}) and (\ref{bm.7}), we get 
\beas Q(f)=\sideset{}{_{\mathbf{p}\in I}}{\sum} S_{\mathbf{p}}(\partial_u(\alpha)) f^{p_0}\left(\partial_u(f)\right)^{p_1} \cdots\big(\partial_u^{k+1}(f)\big)^{p_{k+1}},\eeas
where $\mathbf{p}=(p_0,\ldots,p_{k+1})\in\mathbb{Z}_+^{k+2}$ such that $\sum_{j=0}^{k+1}p_j\leq k+1$, $I$ is finite set of distinct element
and $S_{\mathbf{p}}\big(\partial_u(\alpha)\big)$ is a polynomial in $\partial_u(\alpha)$ with constant coefficients.

Now we consider following two sub-cases.\par

\medskip
{\bf Sub-case 1.2.1.} Let $P(f)\not\equiv 0$. It is clear from (\ref{bm.13}) that every monomial of $P(f)$ has the form
\[\hat{S}_{\mathbf{q}}\big(\partial_u(\alpha)\big) f^{q_0}\big(\partial_u(f)\big)^{q_1} \cdots\big(\partial_u^{k+1}(f)\big)^{q_{k+1}},\]
where $\mathbf{q}=(q_0,\ldots,q_{k+1})\in\mathbb{Z}_+^{k+2}$ such that $\sum_{j=0}^{k+1}q_j=k+1$, $q_0 \leq k$
and $\hat{S}_{\mathbf{q}}\big(\partial_u(\alpha)\big)$ is a polynomial in $\partial_u(\alpha)$ with constant coefficients.
So using Lemma \ref{l2} to (\ref{bm.13}), we get
\bea\label{bm.15}  \parallel\;m\Big(r, \frac{P(f)}{f^k \partial_u(f)}\Big)=o(T(r,f)).\eea

On the other hand using Lemma \ref{l3} to (\ref{bm.10}), we get
\bea\label{bm.16}\parallel\;T(r,P(f))=m(r,P(f))=o(T(r, f)).\eea

Therefore (\ref{bm.15}) and (\ref{bm.16}) yield
\bea\label{bm.17} \parallel\;m\Big(r, \frac{1}{f^{k}\partial_u(f)}\Big)&\leq & m\Big(r, \frac{P(f)}{f^{k}\partial_u(f)}\Big)+m\Big(r, \frac{1}{P(f)}\Big) \nonumber \\
& \leq & m \Big(r, \frac{P}{f^k\partial_u(f)}\Big)+T(r,P(f))+O(1)=o(T(r, f)).\eea

Now using Lemma 2.2 and (\ref{bm.17}), we get
\beas \parallel\;m\Big(r, \frac{1}{F-1}\Big)\leq m\Big(r, \frac{\partial_u(F)}{F-1}\Big)+m\Big(r,\frac{1}{f^{k}\partial_u(f)}\Big)=o(T(r, f))\eeas
and so using Lemma \ref{l2} to (\ref{bm.8}), we get
\[\parallel\;m\big(r,e^{\alpha}\big)\leq m\Big(r,\frac{\partial_u^k(F)}{F-1}\Big)+m\Big(r,\frac{1}{F-1}\Big)+O(1)\leq o(T(r, f)),\]
which shows that 
\[\parallel T\big(r,e^{\alpha}\big)=o(T(r, f)).\]

Using (\ref{bm.5}) to (\ref{bm.8}), we see that
\bea\label{bm.19} e^{\alpha}-1&=&\frac{\partial_u^k(F)-F}{F-1}\\&=&\frac{(k+1)!f(\partial_u(f))^{k} + \frac{k(k-1)}{4}(k+1)!f^{2}(\partial_u(f))^{k-2}\partial_u^2(f)+R_1(f)-f^{k+1}}{f^{k+1}-1},\nonumber\eea
where $R_{1}(f)$ is a differential polynomial in $f$ and each term of $R_{1}(f)$ contains $f^{l} (3 \leq l \leq k$) as a factor.

We assume that $\mu^0_{f}(z_0)> 0$ for $z_0\in \mathbb{C}^m$. Then from (\ref{bm.19}), we have generically
$\mu^1_{e^{\alpha}}(z_0)>0$ over $\text{supp}\;\mu^1_{e^{\alpha}}$. So 
\[\parallel N(r,0;f)\leq N(r,1;e^{\alpha}).\]

Since $\parallel T\big(r,e^{\alpha}\big)=o(T(r, f))$, using first main theorem, we get
\[\parallel\;N_{1)}(r,0;f)\leq N(r,0;f)\leq N(r,1; e^{\alpha})\leq T\big(r,e^{\alpha}\big)+O(1)=o(T(r, f))).\]

Therefore from (\ref{bm.6}), we get $\parallel T(r,f)=o(T(r,f))$, which is impossible.\par

\medskip
{\bf Sub-case 1.2.2.} Let $P(f)\equiv 0$. Now from (\ref{bm.13}), we have 
\bea\label{bm.20} k(k + 1)!(\partial_u(f))^{k + 1}&\equiv& -\frac{k(k+1)(k-3)}{4}(k + 1)!f(\partial_u(f))^{k - 1}\partial_u^2(f)\\&&-(k+1)!f(\partial_u(f))^{k}\partial_u(\alpha)+R_3(f)\nonumber.\eea

Let  
\bea\label{bm.21} h=\frac{\partial_u(f)}{f}.\eea

Suppose $\mu^0_{f)1}(z_0)> 0$ for $z_0\in \mathbb{C}^m$. Then (\ref{bm.20}) yields $\mu^0_{\partial_u(f)}(z_0)>0$ over $\text{supp}\;\mu_{\partial_u(f)}^0$. Now from (\ref{bm.21}), we get generically
$\mu^0_{h}(z_0)\geq 0$ over $\text{supp}\;\mu^0_{h}$ and so  
\[N(r,h)\leq N_{(2}(r,0;f).\]

Using Lemma \ref{l2} and (\ref{bm.5a}) to (\ref{bm.21}), we get
\[\parallel\;T(r,h)=m\Big(r,\frac{\partial_u(f)}{f}\Big)+N\Big(r,\frac{\partial_u(f)}{f}\Big)\leq N_{(2}(r,0;f)+o(T(r,f))=o(T(r,f))\]
and so $h$ is a small function of $f$. Therefore $\partial_u(f)=hf$, where $h$ is a small function of $f$. Note that \[\partial_u^2(f)=\partial_u(h)f+h\partial_u(f)=(\partial_u(h)+h^2)f=h_2f,\]
where $h_2=\partial_u(h)+h^2$ is a small function of $f$. Therefore in general, we have
\bea\label{bm.22} \partial_u^i(f)=h_if,\eea
where $h_i$ is a small function of $f$ for $i=1,2,\ldots$ and $h_1=h$.

Note that 
\[\partial_u(F)=(k+1) f^k\partial_u(f)=(k+1)h_1f^{k+1}.\]

Also using (\ref{bm.22}) to (\ref{bm.5}), we can write 
$\partial^k_u(F)=af^{k+1}$, where $a$ is a small function of $f$. Since $F\not\equiv \partial^k_u(F)$, it follows that $a\not\equiv 1$.
Therefore 
\[\partial_u^k(F)-F=(a-1)F.\]

Consequently from (\ref{bm.8}), we get
\bea\label{bm.23} e^{\alpha}-1&=&\frac{\partial_u^k(F)-F}{F-1}=(a-1)\frac{F}{F-1}.\eea

Now using Lemma \ref{l2a} to (\ref{bm.23}), we have 
\[\parallel\;T(r,e^{\alpha})+O(1)=(k+1)T(r,f)+O(T(r,a))=(k+1)T(r,f)+o(T(r,f)),\] 
which shows that $\parallel T(r,a)=o(T(r,e^{\alpha}))$ and so $a$ is a small function of $e^{\alpha}$.

\smallskip
Let $\mu^1_{e^{\alpha}}(z_1)> 0$ for some $z_1\in \mathbb{C}^m$. If $\mu^{0}_a(z_1)=0$, then from (\ref{bm.23}), we have generically $(k+1)\mu^0_{f}(z_1)=\mu^1_{e^{\alpha}}(z_1)$ over $\text{supp}\;\mu^1_{e^{\alpha}}$. Since $k\geq 2$, we conclude that
\[\parallel\;\ol N(r,1;e^{\alpha})\leq \frac{1}{k+1}N(r,1;e^{\alpha})+N(r,0;a).\]

Therefore using first main theorem and Lemma \ref{l1}, we get
\beas
\parallel\;T\Big(r,e^{\alpha}\Big)&\leq& \ol N(r,e^{\alpha})+\ol N(r,0;e^{\alpha})+\ol N(r,1;e^{\alpha})+o(T(r,e^{\alpha})) \nonumber \\
&\leq &\frac{1}{k+1} T(r,e^{\alpha})+o(T(r,e^{\alpha})),\eeas 
which is impossible.\par

\medskip
{\bf Case 2.} Let $\varphi\equiv 0$. Then from (\ref{bm.1}), we have either $\partial_u(F)\equiv 0$ or $F\equiv \partial^k_u(F)$.
Note that 
\[\partial_u(F)=nf^{n-1}\partial_u(f).\]

If $\partial_u(F)\equiv 0$, then since $f\not\equiv 0$, we have $\partial_u(f)\equiv 0$ and so from (\ref{bm.1a}), we get 
$\partial^k_u(F)\equiv 0$, which is a contradiction. Hence $F\equiv \partial^k_u(F)$ and so (\ref{bm.1a}) gives
\bea\label{bm.24} f^{n-k}&\equiv&
\frac{n!}{(n-k)!}(\partial_u(f))^{k}+\frac{k(k-1)}{2}\frac{n!}{(n-k+1)!}f(\partial_u(f))^{k-2}\partial_u^2(f)+\cdots\nonumber\\&&+nf^{k-1}\partial_u^k(f).\eea

Let 
\[\psi_1=\frac{\partial_u(f)}{f}.\]

Assume that $\mu^0_{f}(z_0)> 0$ for $z_0\in \mathbb{C}^m$. Then from (\ref{bm.24}), we get $\mu^0_{f}(z_0)\leq \mu^0_{\partial_u(f)}(z_0)$ over $\text{supp}\;\mu_{\partial_u(f)}^0$. So we have generically
$\mu^0_{\psi_1}(z_0)\geq 0$ over $\text{supp}\;\mu^0_{\psi_1}$.
Now using Lemma \ref{l2}, we deduce that $\psi_1$ is an entire small function of $f$. Note that
\bea\label{bm.25} \partial_u(F)=nf^{n-1}\partial_u(f)=n\psi_1F.\eea

\smallskip
First we suppose that $k=1$. Since $F\equiv \partial_u(F)$, from (\ref{bm.15}) we have $F\equiv n\psi_1F$ and so $\psi_1\equiv \frac{1}{n}$. In this case we get $f\equiv n\partial_u(f)$.

\smallskip
Next we suppose that $k\geq 2$. From (\ref{bm.25}), we certainly have $\partial^2_u(F)=\psi_2F$, where $\psi_2=n\partial_u(\psi_1)+n^2\psi_1^2$. Now by induction and using (\ref{bm.25}) repeatedly, we obtain $\partial_u^k(F)=\psi_k F$. Here $\psi_i$ satisfies the recurrence formula
\bea\label{bm.26} \psi_{i+1}=n\partial_u(\psi_i)+n\psi_1\psi_i,\;\;i=1,2,\ldots,k-1.\eea

From the recurrence formula (\ref{bm.26}) for $\psi_i$, we can easily derive the expression
\bea\label{bm.27} \psi_k=\psi^{k}+\tilde P(\psi_1),\eea
where 
\[\tilde P(\psi_1)=\sum_{\mathbf{p}\in I} S_{\mathbf{p}}\psi_1^{p_0}\left(\partial_u(\psi_1)\right)^{p_1} \cdots\big(\partial_u^{k-1}(\psi_1)\big)^{p_{k-1}},\]
$\mathbf{p}=(p_0,\ldots,p_{k-1})\in\mathbb{Z}_+^{k}$ such that $\sum_{j=0}^{k-1}p_j\leq k-1$, $I$ is finite set of distinct element
and $S_{\mathbf{p}}$ is a constant.

Since $\partial_u^k(F)=\psi_k F$ and $F\equiv \partial^k_u(F)$, it follows that $\psi_k\equiv 1$ and so from (\ref{bm.27}), we get 
\bea\label{bm.28} \psi_1^{k-1}.\psi_1=-\tilde P(\psi_1)+1.\eea

Now applying Lemma \ref{l3} to (\ref{bm.28}), we get $\parallel m(r,\psi_1)=o(T(r,\psi_1))$.
Since $N(r,\psi_1)=0$, we have $\parallel T(r,\psi_1)=o(T(r,\psi_1))$, which shows that $\psi_1$ is a constant.
Therefore $\partial^k_u(F)=n^k\psi_1^kF$. Since $F\equiv \partial^k_u(F)$, we have $n^k\psi_1^k=1$ and so $f\equiv c\partial_u(f)$, where $(c/n)^k=1$.

\smallskip
Thus in either case we have $f\equiv c\partial_u(f)$, where $(c/n)^k=1$. In particular if $u_j=1$ and $u_i=0$ for $i\neq j$, then $f\equiv c\partial_{z_j}(f)$ and so
\[f(z)=e^{cz_j+\alpha(z)},\]
where $\alpha(z)=\alpha(z_1,\ldots,z_{j-1}, z_{j+1},\ldots,z_m)$ is a non-constant entire function in $\mathbb{C}^{m-1}$.
This completes the proof.
\end{proof}

\begin{proof}[Proof of Theorem \ref{t2.2}] 
Let $F=f^n$ and 
\bea\label{bbm.1} \varphi=\frac{\partial_u(F)(F-\partial_u(F))}{F(F-1)}=\frac{n f^{n-2}\partial_u(f)(f-n\partial_u(f))}{f^n-1}.\eea

Now we consider following two cases.\par

\medskip
{\bf Case 1.} Let $\varphi\not\equiv 0$. Using Lemma \ref{l2} to (\ref{bbm.1}), we get $\parallel m(r,\varphi)=o(T(r,f))$. 

Let $\mu_F^1(z_0)>0$ for $z_0\in \mathbb{C}^m$. Since $\operatorname{supp} \mu_F^1=\operatorname{supp}\mu_{\partial_u(F)}^1$, we have $\mu_{\partial_u(F)}^1(z_0)>0$. If $\mu_{F(2}^1(z_0)>0$, then $\mu_{\partial_u(F)}^0(z_0)>0$. Since $F$ and $\partial_u(F)$ share $1$ IM, we have $\mu_{\partial_u(F)}^1(z_0)>0$. So we get a contradiction. Hence $\mu_{F(2}^1(z_0)=0$ and so  
\bea\label{bbm.2} N_{(2}(r,1;F)=0.\eea

If $\mu_{F1)}^1(z_0)>0$, then from (\ref{bbm.1}) we have generically $\mu_{\varphi}^0(z_0)\geq 0$ over $\text{supp}\;\mu_F^1$.
Hence $\mu_{\varphi}^{\infty}(z)=0$ for all $z\in\mathbb{C}^m$ and so $N(r,\varphi)=0$. Since $\parallel m(r,\varphi)=o(T(r,f))$, we get $\parallel T(r,\varphi)=o(T(r,f))$. Using Lemma \ref{l2} to (\ref{bbm.1}), we get $\parallel m(r,0;f)=o(T(r,f))$.\par

Now we consider following two sub-cases.\par

\medskip
{\bf Sub-case 1.1} Let $n\geq 3$. Take $z_0\in \mathbb{C}^m$.
Let $\mu_f^0(z_0)>0$. Clearly $\mu_F^0(z_0)=n\mu^0_f(z_0)$ over $\text{supp}\;\mu_f^0$. Now 
$\mu_{\partial_u^k(F)}^0(z_0)\geq (n-1)\mu_f^0(z_0)$ over $\text{supp}\;\mu_f^0$. Then from (\ref{bbm.1}) we get generically
$\mu_{\varphi}^0(z_0)\geq (n-2)\mu_f^0(z_0)$ over $\text{supp}\;\mu_f^0$.

Now proceeding in the same way as done in the proof of Sub-case 1.1 of Theorem \ref{t2.1}, we get a contradiction.\par

\medskip
{\bf Sub-case 1.2.} Let $n=2$. In this case also one can easily prove that
\bea\label{bbm.2a}\parallel\;N_{(2}(r,0;f)=o(T(r,f))\;\;\text{and}\;\;\parallel\;T(r,f)=N_{1)}(r,0;f)+o(T(r,f)).\eea

Now from (\ref{bbm.1}), we get 
\[-(\partial_u(F))^2+F\partial_u(F)=\varphi (F^2-F).\]

Clearly we have
\bea\label{bbm.3}-2\partial_u(F) \partial_u^2(F)+(\partial_u(F))^2+F \partial_u^2(F)=\partial_u(\varphi) (F^2-F)+\varphi (2 F \partial_u(F)-\partial_u(F))\eea
and
\bea\label{bbm.4}&&-2(\partial_u^2(F))^2-2 \partial_u(F) \partial_u^3(F)+3\partial_u(F) \partial_u^2(F)+F \partial_u^3(F)=\partial_u^2(\varphi)(F^2-F)\nonumber\\&&
+2\partial_u(\varphi)(2F \partial_u(F)-\partial_u(F))+\varphi (2(\partial_u(F))^2+2 F \partial_u^2(F)-\partial_u^2(F)).\eea

Assume that $\mu_{F1)}^1(z_1)>0$ for some $z_1\in\mathbb{C}^m$. By the given condition, we have $\mu_{\partial_u(F)}^1(z_1)>0$. It is easy to conclude respectively from (\ref{bbm.3}) and (\ref{bbm.4}) that
\bea\label{bbm.5} \mu_{\partial_u^2(F)-\varphi-1}^0(z_1)>0\eea
and
\bea\label{bbm.5a}\mu_{\partial_u^3(F)-2\partial_u(\varphi)+\varphi^2-2\varphi-1}^0(z_1)>0.\eea

\smallskip
If possible suppose $\partial_u^2(F)-(\varphi+1) \partial_u(F)\equiv 0$. Then we get
\bea\label{bbm.6}(\partial_u(f))^2\equiv ((\varphi+1)\partial_u(f)-\partial^2_u(f))f.\eea

Assume that $\mu_{f1)}^0(z_0)>0$ for some $z_0\in\mathbb{C}^m$. Then from (\ref{bbm.6}), we get $\mu^0_{\partial_u(f)}(z_0)>0$ over $\text{supp}\;\mu_{\partial_u(f)}^0$. Consequently if 
\[h=\frac{\partial_u(f)}{f},\]
then we have generically
$\mu^0_{h}(z_0)\geq 0$ over $\text{supp}\;\mu^0_{h}$ and so using Lemma \ref{l2} and (\ref{bbm.2a}), one can easily deduce that
$h$ is a small function of $f$. Therefore $\partial_u(f)=hf$, where $h$ is a small function of $f$. Now from (\ref{bbm.1}), we get
\[\varphi=\frac{2h(1-2h)f^2}{f^2-1}\]
and so by Lemma \ref{l2a}, we get $\parallel T(r,f)=o(T(r,f))$, which is impossible. 

\smallskip
Hence $\partial_u^2(F)-(\varphi+1) \partial_u(F)\not\equiv 0$.
Let 
\bea\label{bbm.7} \phi=\frac{\partial_u^2(F)-(\varphi+1)\partial_u(F)}{F-1}.\eea

Clearly $\phi\not\equiv 0$ and by Lemma \ref{l2}, we get $\parallel m(r,\phi)=o(T(r,f))$. Let $\mu_{F1)}^1(z_1)>0$ for some $z_1\in\mathbb{C}^m$. Then from (\ref{bbm.5}) and (\ref{bbm.7}), we have generically $\mu_{\phi}^0(z_0)\geq 0$ over $\text{supp}\;\mu_F^1$. Therefore from (\ref{bbm.2}), we deduce that $\mu_{\phi}^{\infty}(z)=0$ generically for all $z\in\mathbb{C}^m$. Consequently $N(r,\phi)=0$. Since $m(r,\phi)=o(T(r,f))$, we have $T(r,\phi)=o(T(r,f))$.
Now from (\ref{bbm.7}), we get
\bea\label{bbm.8}\partial_u^2(F)=(\varphi+1) \partial_u(F)+\phi(F-1)\eea
and so 
\[\partial_u^3(F)=\partial_u(\varphi)\partial_u(F)+(\varphi+1)\partial_u^2(F)+\partial_u(\phi)(F-1)+\phi \partial_u(F).\]

Using (\ref{bbm.8}), we get  
\bea\label{bbm.9}\partial_u^3(F)=(\varphi^2+2\varphi+\phi+\partial_u(\varphi)+1)\partial_u(F)+\left((\varphi+1)\phi+\partial_u(\phi)\right)(F-1).\eea

We now consider following two sub-cases.\par

\medskip
{\bf Sub-case 1.2.1.} Let $2\varphi^2-\partial_u(\varphi)+\phi\not\equiv 0$.
If possible suppose 
\[\partial_u^3(F)-\left(2 \partial_u(\varphi)-\varphi^2+2 \varphi+1\right) \partial_u(F)\equiv 0.\]

Then from (\ref{bbm.9}), we get
\bea\label{bbm.10}(2\varphi^2+\phi-\partial_u(\varphi))\partial_u(F)+\left((\varphi+1)\phi+\partial_u(\phi)\right)(F-1)\equiv 0.\eea

Since $\partial_u(F)\not\equiv 0$, from (\ref{bbm.10}), we get 
\[(\varphi+1)\phi+\partial_u(\phi)\not\equiv 0.\]

Let $\mu_{f1)}^0(z_0)>0$ for some $z_0\in\mathbb{C}^m$. Then from (\ref{bbm.10}), we get $\mu^0_{(\varphi+1)\phi+\partial_u(\phi)}(z_0)>0$ over $\text{supp}\;\mu_{(\varphi+1)\phi+\partial_u(\phi)}^0$. Consequently by the first main theorem, we get
\[N_{1)}(r,0;f)\leq N(r,0;(\varphi+1)\phi+\partial_u(\phi))\leq T(r,(\varphi+1)\phi+\partial_u(\phi))+O(1)=o(T(r,f))\]
and so from (\ref{bbm.2a}) we have $T(r,f)=o(T(r,f))$, which is impossible. 

Hence 
\[\partial_u^3(F)-\left(2 \partial_u(\varphi)-\varphi^2+2 \varphi+1\right) \partial_u(F)\not\equiv 0.\]

Let
\bea\label{bbm.11} \psi=\frac{\partial_u^3(F)-\left(2\partial_u(\varphi)-\varphi^2+2\varphi+1\right)\partial_u(F)}{F-1}.\eea

Clearly $\psi\not\equiv 0$ and by Lemma \ref{l2}, we get $\parallel m(r,\psi)=o(T(r,f))$. Let $\mu_{F1)}^1(z_1)>0$ for some $z_1\in\mathbb{C}^m$. Then from (\ref{bbm.5a}) and (\ref{bbm.11}), we have generically $\mu_{\psi}^0(z_0)\geq 0$ over $\text{supp}\;\mu_F^1$. Therefore using (\ref{bbm.2}), we get $N(r,\psi)=0$ and so $\parallel T(r,\psi)=o(T(r,f))$.
Now (\ref{bbm.9}) and (\ref{bbm.11}) give
\bea\label{bbm.12}\left(2\varphi^2-\partial_u(\varphi)+\phi\right)\partial_u(F)=\left(\psi-\partial_u(\phi)-(1-\varphi)\phi\right)(F-1).\eea

Since $\partial_u(F)\not\equiv 0$, from (\ref{bbm.12}), we see that 
\[\psi-\partial_u(\phi)-(1-\varphi)\phi\not\equiv 0.\]

Let $\mu_{f1)}^0(z_0)>0$ for some $z_0\in\mathbb{C}^m$. Then $\mu_{\partial_u(F))}^0(z_0)>0$ and so from (\ref{bbm.12}), we get generically $\mu^0_{\psi-\partial_u(\phi)-(1-\varphi)\phi}(z_0)>0$ over $\text{supp}\;\mu_{\psi-\partial_u(\phi)-(1-\varphi)\phi}^0$. Clearly first main theorem gives
\beas \parallel\;N_{1)}(r,0;f)&\leq& N(r,0;\psi-\partial_u(\phi)-(1-\varphi)\phi)\leq T(r,(\psi-\partial_u(\phi)-(1-\varphi)\phi)+O(1)\\&=&o(T(r,f))\eeas
and so from (\ref{bbm.2a}), we have $\parallel T(r,f)=o(T(r,f))$, which is impossible.\par

\medskip
{\bf Sub-case 1.2.2.} Let 
\bea\label{bbm.13} 2\varphi^2-\partial_u(\varphi)+\phi\equiv 0.\eea

Assume that $\mu_{f}^0(z_0)>0$ for some $z_0\in\mathbb{C}^m$. Then $\mu_{\partial_u(F))}^0(z_0)>0$ and so from (\ref{bbm.4}) and (\ref{bbm.8}), we get respectively
\bea\label{bbm.14}\label{bbm.15} \mu^0_{\partial_u^2(F)\left(2\partial_u^2(F)+\varphi\right)}(z_0)>0\;\;\text{and}\;\;\mu^0_{\partial_u^2(F)+\phi}(z_0)>0\eea
over $\text{supp}\;\mu_{f}^0$.
We claim that 
\bea\label{bbm.16}\mu^0_{2\partial_u^2(F)+\varphi}(z_0)>0.\eea

Note that 
\[\partial_u^2(F)=2(\partial_u(f))^2+2f\partial_u(f).\]

If $\mu^0_{\partial_u^2(F)}(z_0)>0$, then we have generically $\mu^0_{\partial_u(f)}(z_0)>0$. Clearly from (\ref{bbm.1}), we get $\mu^0_{\varphi}(z_0)>0$ and so (\ref{bbm.16}) holds. If $\mu^0_{\partial_u^2(F)}(z_0)=0$, then (\ref{bbm.16}) holds directly from (\ref{bbm.14}).\par

\smallskip
Let $\phi-\frac{1}{2}\varphi\not\equiv 0$. It is clear from (\ref{bbm.15}) and (\ref{bbm.16}) that 
$\mu^0_{\phi-\frac{1}{2}\varphi}(z_0)>0$ holds generically. Consequently by the first main theorem, we get
\[\parallel\;N_{1)}(r,0;f)\leq N(r,0; \phi-\frac{1}{2}\varphi )\leq T\Big(r,\phi-\frac{1}{2}\varphi\Big)+O(1)=o(T(r,f))\]
and so from (\ref{bbm.2a}), we have $\parallel T(r,f)=o(T(r,f))$, which is impossible.\par

\smallskip
Let $\phi-\frac{1}{2}\varphi\equiv 0$. Then from (\ref{bbm.13}), we get
$2\varphi^2+\frac{1}{2}\varphi-\partial_u(\varphi)\equiv 0$. Since $\varphi\not\equiv 0$, we have
\bea\label{bbm.17}2\varphi=\frac{\partial_u(\varphi)}{\varphi}-\frac{1}{2}.\eea

Again since $\varphi$ is an entire function, applying Lemma \ref{l2} to (\ref{bbm.17}), we get 
\[\parallel\;T(r,\varphi)=m(r,\varphi)\leq m\Big(r, \frac{\partial_u(\varphi)}{\varphi}\Big)+O(1)=o(T(r,\varphi)),\]
which means that $\varphi$ is a constant and so (\ref{bbm.17}), we have 
$\varphi=-\frac{1}{4}$. 
Now from (\ref{bbm.1}), we get 
\[(2\partial_u(F)-F)^2=F.\]

Note that 
\[\partial_u(F)=2f\partial_u(f)\;\;\text{and}\;\;\partial_u^2(F)=2(\partial_u(f))^2+2f\partial^2_u(f).\]

Clearly
\bea\label{bbm.18} (4\partial_u(f)-f)^2=1.\eea

Now from (\ref{bbm.18}), we get $\partial_u(f)=\frac{f\pm 1}{4}$. 

We consider following sub-cases.\par

\medskip
{\bf Sub-case 1.2.2.1.} Let $\partial_u(f)=\frac{f+1}{4}$. Note that 
\[F-1=(f+1)(f-1)\;\;\text{and}\;\;\partial_u^2(f)=\frac{1}{4}\partial_u(f)=\frac{1}{16}(f+1).\]

Clearly $\partial_u(F)=\frac{1}{2}f(f+1)$ 
and 
\bea\label{bbm.19}\partial_u^2(F)=\frac{1}{8}(f+1)^2+\frac{1}{8}f(f+1)=\frac{1}{8}(2f^2+3f+1).\eea 

Assume that $\mu_{f}^{-1}(z_1)>0$ for some $z_1\in\mathbb{C}^m$. Since $F-1=(f+1)(f-1)$ and $N_{(2}(r,1;F)=0$, it follows that $\mu_{F1)}^{1}(z_1)>0$.
By the given condition, we have $\mu_{\partial_u(F)}^1(z_1)>0$. Again since $\partial_u(f)=\frac{f+1}{4}$, we have
$\mu_{\partial_u(f)1)}^{-1}(z_1)>0$ and so $\mu_{\partial_u(F)}^{0}(z_1)>0$, which is impossible. Therefore $\mu^{-1}_f(z)=0$ generically for all $z\in\mathbb{C}^m$.

Let us assume that $\mu_{F}^{1}(z_2)>0$ for some $z_2\in\mathbb{C}^m$. Certainly $\mu_{f}^{1}(z_2)>0$. Then from (\ref{bbm.19}), we have generically $\mu_{\partial_u^2(F)}^{3/4}(z_2)>0$. Consequently $\mu_{\partial_u(F)(2}^{1}(z)=0$ generically for all $z\in\mathbb{C}^m$ and so $N_{(2}(r,1;\partial_u(F))=0$. This shows that $F$ and $\partial_u(F)$ share $1$ CM. Then by Theorem \ref{t2.1}, we get $F\equiv \partial_u(F)$, which is impossible.\par

\medskip
{\bf Sub-case 1.2.2.2.} Let $\partial_u(f)=\frac{f-1}{4}$. Proceeding in the same way as done in the proof of Sub-case 1.2.2.1, we can easily get a contradiction.\par

\medskip
{\bf Case 2.} Let $\varphi\equiv 0$. Then from (\ref{bbm.1}), we have $F\equiv \partial_u(F)$.
Since $\partial_u(F)=nf^{n-1}\partial_u(f)$, we have $f\equiv n\partial_u(f)$. In particular if $u_j=1$ and $u_i=0$ for $i\neq j$, then we get $f=e^{\alpha}$, where $\alpha$ is a non-constant entire function in $\mathbb{C}^m$ such that $n\partial_{z_j}(\alpha)\equiv 1$.
Hence the proof is complete.
\end{proof}

\begin{proof}[Proof of Theorem \ref{t2.3}] 
Let 
\bea\label{sb.1} H=\frac{\partial_u^k(F)}{F},\eea
where $F=f^n$. We divide following two cases.\par

\medskip
{\bf Case 1.} Let $H\not\equiv 1$. Then using first main theorem and Lemma \ref{l2}, we get
\bea\label{sb.2} \parallel\;\ol N(r,1;F)\leq\ol N\Big(r,\frac{F}{\partial_u^k(F)-F}\Big)
&\leq& T\Big(r,\frac{\partial_u^k(F)-F}{F}\Big)+O(1)\\&\leq& T(r,H)+O(1)\leq  N(r,H)+o(T(r,f)).\nonumber\eea

We now prove that
\bea\label{sb.2a} \mu^{\infty}_H\leq k\mu^0_{f,1}.\eea

Note that $f^{-1}(\{0\})_s$ is an analytic subset of $\mathbb{C}^m$ of dimension $\leq m-2$, where $f^{-1}(\{0\})_s$ is the set of singular points of $f^{-1}(\{0\})$. It suffices to prove (\ref{sb.2a}) on $\mathbb{C}^m-f^{-1}(\{0\})_s$. Take $z_0\in \mathbb{C}^m-f^{-1}(\{0\})_s$. Then there is a holomorphic coordinate system $\left(U ; \varphi_1, \ldots, \varphi_m\right)$ of $z_0$ in $\mathbb{C}^m$ such that $U \cap f^{-1}(\{0\})=\left\{z \in U \mid \varphi_1(z)=0\right\}$ and $\varphi_j(z_0)=0$ for $j=1,2,\ldots,m$ (see proof of Lemma 2.3 \cite{FL}). Therefore biholomorphic coordinate transformation $z_j=z_j(\varphi_1, \ldots, \varphi_m)$, $j=1, \ldots, m$ near $0$ exists such that $z_0=z(0)=(z_1(0), \ldots, z_m(0))$. 
So we can write $f=\varphi_1^{-l} \hat{f}\left(\varphi_1, \ldots, \varphi_m\right)$, 
where $\hat{f}$ is a holomorphic function near $0$ such that $\hat{f}$ is zero free along the set $f^{-1}(\{0\})$. 
Clearly 
\bea\label{sb.2aa} F=\varphi_1^{nl} \tilde {f}\left(\varphi_1, \ldots, \varphi_m\right),\eea
where $\tilde {f}\left(\varphi_1, \ldots, \varphi_m\right)=\hat{f}^n\left(\varphi_1, \ldots, \varphi_m\right)$. Now for any $i\in\mathbb{Z}[1, m]$, we have
\bea\label{sb.2bb}\frac{\partial F}{\partial z_i}=\sideset{}{_{j=1}^m}{\sum} \frac{\partial F}{\partial \varphi_j}\frac{\partial \varphi_j}{\partial z_i}=nl \varphi_1^{nl-1}\tilde f\frac{\partial \varphi_1}{\partial z_i}+\varphi_1^{nl}\sideset{}{_{j=1}^m}{\sum}\frac{\partial \tilde f}{\partial \varphi_j}\frac{\partial \varphi_j}{\partial z_i}\eea
and so
\beas \partial_u(F)=nl \varphi_1^{nl-1}\tilde f \partial_h(\varphi_1) +\varphi_1^{nl}\sideset{}{_{j=1}^m}{\sum}\frac{\partial \tilde f}{\partial \varphi_j}\partial_u(\varphi_j).\eeas
which means $\mu^{0}_{\partial_u(F)}(z_0)\leq nl-1=n\mu^{0}_f(z_0)-\mu^{0}_{f,1}(z_0)$.
By induction, we can show
\bea\label{sb.2b} \mu^{0}_{\partial_u^k(F)}(z_0)\leq n\mu^{0}_f(z_0)-k\mu^{0}_{f,1}(z_0).\eea

Therefore from (\ref{sb.1}), (\ref{sb.2aa}) and (\ref{sb.2b}), we get $\mu^{\infty}_H(z_0)\leq k\mu^{\infty}_{f,1}(z_0)$ and so inequality (\ref{sb.2a}) holds. Consequently
$N(r,H)\leq k\ol N(r,0;f)$ and so from (\ref{sb.2}), we have
\bea\label{sb.3} \parallel\;\ol N(r,1;F)\leq k\ol N(r,0;f)+o(T(r,f)).\eea
 
Now using (\ref{sb.3}), we get by Lemmas \ref{l1} and \ref{l2a} that
\beas n\;T(r,f)=T(r,F)+O(1)&\leq& \ol N(r,F)+\ol N(r,0;f)+\ol N (r,1;F)+o(T(r,f))
\\ &\leq& (k+1)\ol N(r,0;f)+o(T(r,f)),\eeas 
which is impossible, since $n\geq k+2$.\par

\medskip
{\bf Case 2.} Let $H\equiv 1$. Then $F\equiv \partial_u^k(F)$. Remaining part of the proof of Theorem \ref{t2.3} follows from the proof of Case 2 of Theorem \ref{t2.1}.
This completes the proof.
\end{proof}

\begin{proof}[Proof of Corollary \ref{c2.1}]
Let $H$ be defined as in (\ref{sb.1}). By the given conditions, we see that $F$ and $\partial_u^k(F)$ share $1$ IM and $n\geq k$.
Suppose $H\not\equiv 1$. Since $\ol N_{2)}(r,0;f)=o(T(r, f))$, using (\ref{sb.3}), we get by Lemmas \ref{l1} and \ref{l2a} that
\beas n\;T(r,f)\leq \ol N(r,0;f)+\ol N(r,1;F)+o(T(r,f))&\leq& (k+1)\ol N(r,0;f)+o(T(r,f))
\\& \leq & \frac{k + 1}{3} T(r, f) + o(T(r, f)),\eeas 
which is impossible, since $n\geq k$.

\smallskip
If $H\equiv 1$, then the proof of Theorem \ref{t2.3} follows from the proof of Case 2 of Theorem \ref{t2.1}.
This completes the proof.
\end{proof}

\begin{proof}[Proof of Theorem \ref{t2.4}] 
Let $G=\partial_u^k(F)$. 
We consider following cases.\par

\medskip
{\bf Case 1.} Let $F-1$ and $G-1$ be linearly independent. By Corollary 1.40 \cite{7a}, there is $i \in \mathbb{Z}[1, m]$ such that the Wronskian determinant $W\not\equiv 0$.
Let
\bea\label{sbb.1} S=\frac{W}{(F-1)(G-1)}=\frac{\partial_{z_i}(G)}{G-1}-\frac{\partial_{z_i}(F)}{F-1}.\eea

Define 
\[I=I_F\cup I_G\;\;\text{and}\;\; S=\sideset{}{_{a\in\{0,1,\infty\}}}{\bigcup}(\operatorname{supp}\mu_F^a)_s\cup(\operatorname{supp}\mu_G^a)_s.\]

Clearly $\dim (I\cup S)\leq m-2$. 
First we prove that $(n-k-1)\mu^0_{f,1}\leq \mu^0_S$. It is sufficient to it on $\mathbb{C}^m-(I\cup S)$.
Take $z_0\in\mathbb{C}^m-(I\cup S)$. Let $l=\mu_f^0(z_0)>0$. Then (\ref{sb.2bb}) gives 
\[\mu^0_{\partial_{z_i}(F)}(z_0)\leq n\mu^0_f(z_0)-1.\]

Again from (\ref{sb.2b}), we get 
\[\mu^0_G(z_0)\leq n\mu^{0}_f(z_0)-k\mu^{0}_{f,1}(z_0)\]
and so 
\[\mu^0_{\partial_{z_i}(G)}(z_0)\leq n\mu^{0}_f(z_0)-(k+1)\mu^{0}_{f,1}(z_0).\]

Therefore from (\ref{sbb.1}), we have generically $(n-k-1)\mu^{0}_{f,1}\leq \mu^0_S$. Consequently
\bea\label{sbb.2} (n-k-1)\ol N(r,0;f)\leq N(r,0;S).\eea

\smallskip
Next we prove that 
\bea\label{sbb.2a} \mu^{\infty}_S\leq \mu^{\infty}_{f,1}.\eea

It suffices to prove (\ref{sbb.2a}) on $\mathbb{C}^m-(I\cup S)$. Take $z_0\in\mathbb{C}^m-(I\cup S)$.\par

\smallskip
Let $l_1=\mu_f^{\infty}(z_0)>0$. Since $z_0 \notin S$, there is a holomorphic coordinate system $\left(U ; \varphi_1, \ldots, \varphi_m\right)$ of $z_0$ in $\mathbb{C}^m-(I \cup S)$ such that $U \cap \operatorname{supp} \mu_f^{\infty}=\left\{z \in U \mid \varphi_1(z)=0\right\}$ and $\varphi_j(z_0)=0$ for $j=1,2,\ldots,m$. A biholomorphic coordinate transformation $z_j=z_j(\varphi_1, \ldots, \varphi_m)$, $j=1, \ldots, m$ near $0$ exists such that $z_0=(z_1(0), \ldots, z_m(0))$ and $f=\varphi_1^{-l_1} \hat{f}\left(\varphi_1, \ldots, \varphi_m\right)$, where $\hat{f}$ is a holomorphic function near $0$ and zero free along the set $\operatorname{supp} \mu_f^{\infty}$. Clearly 
\bea\label{sbb.3}F=\varphi_1^{-nl_1} \tilde {f}\left(\varphi_1, \ldots, \varphi_m\right),\eea
where $\tilde {f}\left(\varphi_1, \ldots, \varphi_m\right)=\hat{f}^n\left(\varphi_1, \ldots, \varphi_m\right)$.
Now for any $i\in\mathbb{Z}[1, m]$, we have
\bea\label{sbb.3a} \partial_i(F)=\sideset{}{_{j=1}^m}{\sum} \frac{\partial F}{\partial \varphi_j}\frac{\partial \varphi_j}{\partial z_i}=\frac{nl_1}{\varphi_1^{nl_1+1}}\tilde f\frac{\partial \varphi_1}{\partial z_i}+\frac{1}{\varphi_1^{nl_1}}\sideset{}{_{j=1}^m}{\sum}\frac{\partial \tilde f}{\partial \varphi_j}\frac{\partial \varphi_j}{\partial z_i}\eea
and so
\beas\partial_u(F)=\frac{nl_1}{\varphi_1^{nl+1}}\tilde f \partial_u(\varphi_1) +\frac{1}{\varphi_1^{nl_1}}\sideset{}{_{j=1}^m}{\sum}\frac{\partial \tilde f}{\partial \varphi_j}\partial_u(\varphi_j),\eeas
which means $\mu^{\infty}_{\partial_u(F)}(z_0)\leq nl_1+1=n\mu^{\infty}_f(z_0)+\mu^{\infty}_{f,1}(z_0)$.
By induction, we can show
\bea\label{sbb.4} \mu^{\infty}_G(z_0)=\mu^{\infty}_{\partial_u^k(F)}(z_0)\leq n\mu^{\infty}_f(z_0)+k\mu^{\infty}_{f,1}(z_0).\eea

Consequently 
\bea\label{sbb.5}\mu^{\infty}_{\partial_i(G)}(z_0)\leq n\mu^{\infty}_f(z_0)+(k+1)\mu^{\infty}_{f,1}(z_0).\eea

Now using (\ref{sbb.3})-(\ref{sbb.5}) to (\ref{sbb.1}), we have $\mu^{\infty}_S\leq \mu^{\infty}_{f,1}$ and so the inequality (\ref{sbb.2a}) holds.

\smallskip
Let $\mu_F^1(z_0)>0$. Since $\operatorname{supp} \mu_F^1=\operatorname{supp}\mu_G^1$, we have $\mu_G^1(z_0)>0$.
By the given condition, we get $\mu^1_F(z_0)=\mu^1_G(z_0)$. Then (\ref{sbb.1}) gives $\mu^{\infty}_S(z_0)=0$ and so the inequality (\ref{sbb.2a}) holds.

\smallskip
Consequently from (\ref{sbb.2}), we get $N(r,S)\leq \ol N(r,f)$.
On the other hand using Lemma \ref{l2} to (\ref{sbb.1}), we get $m(r,S)=o(T(r,f))$. Now using first main theorem, we get from (\ref{sbb.2a}) that
\bea\label{sbb.7} \parallel\; (n-k-1)N(r,0;f)\leq N(r,0;S)\leq T(r,S)+O(1)&=&m(r,S)+N(r,S)+O(1)\nonumber\\&\leq& \ol N(r,f)+o(T(r,f)).\eea

Clearly $H\not\equiv 1$, where $H$ is defined as in (\ref{sb.1}). We claim that the following inequality
\bea\label{sbb.8} \mu^{\infty}_H\leq k\mu^{\infty}_{f,1}+k\mu^0_{f,1}\eea
holds. Let $z_0\in\mathbb{C}^m$.

\smallskip
 First we suppose $\mu_f^0(z_0)>0$. Then from the proof of Theorem \ref{t2.3}, we have $\mu^{\infty}_H(z_0)\leq k\mu^0_{f,1}(z_0)$ over $\text{supp}\;\mu^0_f$ and so $\mu^{\infty}_H(z_0)\leq k\mu^0_{f,1}(z_0)$. Then inequality (\ref{sbb.8}) holds.

\smallskip
Next we suppose $l_1=\mu_f^{\infty}(z_0)>0$. Therefore from (\ref{sb.1}), (\ref{sbb.3}) and (\ref{sbb.4}), we get $\mu^{\infty}_H(z_0)\leq k\mu^{\infty}_{f,1}(z_0)$ and so inequality (\ref{sbb.8}) holds. Finally from (\ref{sbb.8}), we get 
\[N(r,H)\leq k\ol N(r,f)+k\ol N(r,0;f)\]
and so (\ref{sb.2}) gives
\bea\label{sbb.9a} \parallel\;\ol N(r,1;F)\leq k\ol N(r,f)+k\ol N(r,0;f)+o(T(r,f)).\eea

Now from (\ref{sbb.7}), we get 
\bea\label{sbb.9} \parallel\;\ol N(r,1;F)\leq k\Big(1+\frac{1}{n-k-1}\Big)\ol N(r,f)+o(T(r,f)).\eea

Then using (\ref{sbb.2}) and (\ref{sbb.9}), we get by Lemmas \ref{l1} and \ref{l2a} that
\beas \parallel\;n\;T(r,f)\leq (k+1)\Big(1+\frac{1}{n-k-1}\Big) T(r, f) + o(T(r, f)),\eeas 
which is impossible, since $n\geq k+2$.\par

\medskip
{\bf Case 2.} Let $F-1$ and $G-1$ be linearly dependent. There exists $c\in\mathbb{C}\setminus\{0\}$ such that $F-1=c(G-1)$.
If $c\neq 1$, then since $n\geq k+2$, we get $N(r,0;f)=0$. Also we have $H\not\equiv 1$. Therefore using (\ref{sbb.9a}), we get by Lemmas \ref{l1} and \ref{l2a} that
\beas \parallel\;n\;T(r,f)\leq \ol 
N(r,F)+\ol N(r,1;F)+o(T(r,f))\leq (k+1)T(r, f) + o(T(r, f)),\eeas 
which is impossible, since $n\geq k+2$. Hence $c=1$ and so $F\equiv \partial_u^k(F)$. In particular if $f$ and $\partial_u(f)$ share $\infty$ CM, then proceeding in the same way as done in the proof of Case 2 of Theorem \ref{t2.1}, we get the require conclusions.
Hence the proof is complete.
\end{proof}

\begin{proof}[Proof of Theorem \ref{t2.5}] 
Let $H$ be defined as in (\ref{sb.1}). If $H\not\equiv 1$, then using (\ref{sbb.9a}), we get by Lemmas \ref{l1} and \ref{l2a} that
\beas \parallel nT(r,f)\leq \ol 
N(r,f)+\ol N(r,0;f)+\ol N(r,1;F)+o(T(r,f))
\leq (2k + 2) T(r, f) + o(T(r, f)),\eeas 
which is impossible, since $n\geq 2k+3$. Hence $H\equiv 1$, i.e., $F\equiv \partial_u(F)$. Remaining part follows from the proof of Case 2 of Theorem \ref{t2.5}. Hence the theorem follows.
\end{proof}

\medskip
{\bf Statements and declarations:}

\smallskip
\noindent \textbf {Conflict of interest:} The authors declare that there are no conflicts of interest regarding the publication of this paper.

\smallskip
\noindent{\bf Funding:} There is no funding received from any organizations for this research work.

\smallskip
\noindent \textbf {Data availability statement:}  Data sharing is not applicable to this article as no database were generated or analyzed during the current study.

\end{document}